\documentclass[final,letterpaper,twoside,12pt]{article}
\usepackage{amsmath}
\usepackage{amssymb}
\usepackage{amsthm}
\usepackage{array}
\usepackage{xy}
\usepackage[pdftex]{graphicx}
\usepackage{hyperref}

\numberwithin{equation}{section}

\newtheorem{thm}{Theorem}[section]
\newtheorem{cor}[thm]{Corollary}
\newtheorem{lem}[thm]{Lemma}

\newtheorem{prop}[thm]{Proposition}
\newtheorem{rem}[thm]{Remark}

\def\signaf{\bigskip \begin{center} {\sc Alessio Figalli\par\vspace{3mm}
Department of Mathematics\\
The University of Texas at Austin\\
1 University Station, C1200\\
Austin TX 78712, USA\\
email:} {\tt figalli@math.utexas.edu} \end{center}}

\def\signei{\bigskip\begin{center} {\sc Emanuel Indrei\par\vspace{3mm}
Department of Mathematics\\
The University of Texas at Austin\\
1 University Station, C1200\\
Austin TX 78712, USA\\
email:} {\tt eindrei@math.utexas.edu}
\end{center}}

\begin{document}

\begin{center}{\large \textbf{A SHARP STABILITY RESULT FOR THE RELATIVE ISOPERIMETRIC INEQUALITY INSIDE CONVEX CONES}}
\end{center} 
\begin{center}
{\large A. Figalli and  E. Indrei}
\end{center}

\begin{abstract}
The relative isoperimetric inequality inside an open, convex cone $\mathcal C$ states that, at fixed volume, $B_r \cap \mathcal C$ minimizes
the perimeter inside $\mathcal C$. Starting from the observation that this result
can be recovered as a corollary of the anisotropic isoperimetric inequality, 
we exploit a variant of Gromov's proof of the classical isoperimetric inequality
to prove a sharp stability result for the relative isoperimetric inequality inside $\mathcal C$.
Our proof follows the line of reasoning in \cite{Fi}, though several new ideas are needed in order to deal with the lack of translation invariance in our problem.
\end{abstract}

\section{Introduction}
\subsection{Overview}
Recently, there has been a lot of interest in quantitative estimates for isoperimetric \cite{FMP, Fi, Fi3, EFT, NF}, Sobolev \cite{Cord, Cia, Bia, Fi5}, Gagliardo-Nirenberg \cite{CF},
and Brunn-Minkowski \cite{Fi2,BMrmk} type inequalities. The aim of all of these results is to show that if a
set/function almost attains the equality in one of these inequalities, then it is close (in a quantitative way) to a minimizer.
These results have natural
applications in the study of the asymptotic behavior of solutions to evolution equations
\cite{CF}, and
to show stability for minimizers of perturbed variational problems, see for instance \cite{Mur, Fi4}.
Our goal is to investigate stability for the relative isoperimetric inequality inside convex cones. 
This inequality has been used, for instance, to characterize isoperimetric regions inside convex polytopes for small volumes \cite[Corollary 3]{morgan}.
Hence, as in \cite{Fi4}, one may use our stability result to prove quantitative closeness to such isoperimetric regions in perturbed situations.

Let $n\geq 2$, and $\mathcal{C} \subset \mathbb{R}^n$ be an open, convex cone. We denote the unit ball in $\mathbb{R}^n$ centered at the origin by $B_1$
(similar notation is used for a generic ball) and the De Giorgi perimeter of $E$ relative to $\mathcal{C}$ by
\begin{equation} \label{i11}
P(E | \mathcal{C}) := \displaystyle \sup \biggl \{\int_{E} \operatorname{div} \psi dx: \psi \in C_0^{\infty}(\mathcal{C}; \mathbb{R}^n), \hskip .05in |\psi| \leq 1 \biggr \}.
\end{equation}
The relative isoperimetric inequality for convex cones states that if $E \subset \mathcal{C}$ is a Borel set with finite Lebesgue measure $|E|$, then   
\begin{equation} \label{i1}
n|B_1 \cap \mathcal{C}|^{\frac{1}{n}} |E|^{\frac{n-1}{n}} \leq P(E | \mathcal{C}).
\end{equation}

If $E$ has a smooth boundary, the perimeter of $E$ is simply the $(n-1)$-Hausdorff measure of the boundary of $E$
inside the cone (i.e. $P(E | \mathcal{C})=H^{n-1}(\partial E \cap \mathcal{C})$). We also note that if one replaces $\mathcal{C}$ by $\mathbb{R}^n$,
then the above inequality reduces to the classical isoperimetric inequality for which there are many different proofs and formulations (see e.g. \cite{Os},
\cite{GT}, \cite{IC}, \cite{NF}, \cite{BZ}, \cite{BKK}). However, (\ref{i1}) is ultimately due to Lions and Pacella \cite{LP} (see also \cite{RR} for a different proof using
secord order variations). Their proof is based on the Brunn-Minkowski inequality which states that if $A,B \subset \mathbb{R}^n$ are measurable, then
\begin{equation} \label{b1} 
|A+B|^{\frac{1}{n}} \geq |A|^{\frac{1}{n}}+|B|^{\frac{1}{n}}.
\end{equation}

As we will show below, (\ref{i1}) can be seen as an immediate corollary of the anisotropic isoperimetric inequality \eqref{b2}.
This fact suggested to us that there should also be a direct proof of (\ref{i1}) using optimal transport theory (see Theorem \ref{7}), as is the case for the
anisotropic isoperimetric inequality \cite{Cord,Fi}.\footnote{After completion of this work, we discovered that Frank Morgan had already observed that Gromov's argument may be used
to prove the relative isoperimetric inequality inside convex cones \cite[Remark after Theorem 10.6]{morganbook}, though he was thinking about using the Knothe map
instead of the Brenier map. However, as observed in \cite[Section 1.5]{Fi}, the Brenier map is much more powerful than the Knothe map when dealing with stability estimates.} 
The aim of this paper (in the spirit of \cite{Fi}) is to exploit such a proof in order to establish a quantitative version of (\ref{i1}). To make this precise, we need some more notation.  

We define the \textit{relative isoperimetric deficit} of a Borel set $E$ by  
\begin{equation} \label{3}
\mu(E):=\frac{P(E | \mathcal{C})}{n|B_1 \cap \mathcal{C}|^{\frac{1}{n}}|E|^{\frac{n-1}{n}}} -1.
\end{equation}       
Note that (\ref{i1}) implies $\mu(E)\geq 0$. The equality cases were considered in \cite{LP} for the special case when $\mathcal{C} \setminus \{0\}$ is smooth (see also \cite{RR}).
We will work out the general case in Theorem \ref{7} with a self-contained proof. However, the (nontrivial) equality case is not needed in proving the following theorem (which is in any case a much stronger statement):
\begin{thm} \label{t0}
Let $\mathcal{C} \subset \mathbb{R}^n$ be an open, convex cone containing no lines, $K=B_1 \cap \mathcal{C}$, and $E \subset \mathcal{C}$ a set of finite perimeter with $0<|E|<\infty$. Suppose $s>0$ satisfies $|E|=|sK|$. Then there exists a constant $C(n,K) > 0$ such that $$ \displaystyle \frac{|E \Delta (sK)|}{|E|} \leq C(n,K) \sqrt{\mu(E)}.$$
\end{thm}  

The assumption that $\mathcal{C}$ contains no lines is crucial. To see this, consider the extreme case when $\mathcal{C}=\mathbb{R}^n$. Let $\nu \in \mathbb{S}^{n-1}$ be any unit vector and set $E=2\nu+B_1$ so that $|E \Delta B_1|=2|B_1|>0$.  However, $\mu(E)=0$ so that in this case Theorem \ref{t0} can only be true up to a translation, and this is precisely the main result in \cite{Fi} and \cite{FMP}. Similar reasoning can be applied to the case when $\mathcal{C}$ is a proper convex cone containing a line (e.g. a half space). Indeed, if $\mathcal{C}$ contains a line, then by convexity one can show that (up to a change of coordinates) it is of the form $\mathbb{R} \times \mathcal{\tilde{C}}$, with $\mathcal{\tilde{C}}\subset \mathbb{R}^{n-1}$ an open, convex cone. Therefore, by taking $E$ to be a translated version of $K$ along the first coordinate, the symmetric difference will be positive, whereas the relative deficit will remain $0$. 

In general (up to a change of coordinates), every convex cone is of the form $\mathcal{C}= \mathbb{R}^{k} \times \mathcal{\tilde{C}}$, where $\mathcal{\tilde{C}} \subset \mathbb{R}^{n-k}$ is a convex cone containing no lines. Indeed, Theorem \ref{t0} follows from our main result: 

\begin{thm} \label{t1}
Let $\mathcal{C}=\mathbb{R}^{k} \times \mathcal{\tilde{C}}$, where $k \in \{0,\ldots,n\}$ and $\mathcal{\tilde{C}} \subset \mathbb{R}^{n-k}$ is an open, convex cone containing no lines. Set $K=B_1 \cap \mathcal{C}$, and let $E \subset \mathcal{C}$ be a set of finite perimeter with $0<|E|<\infty$. Suppose $s>0$ satisfies $|E|=|sK|$. Then there exists a constant $C(n,K) > 0$ such that $$ \inf \bigg\{\displaystyle \frac{|E \Delta (sK+x)|}{|E|}: x=(x_1,\ldots,x_k,0,\ldots,0) \bigg \} \leq C(n,K) \sqrt{\mu(E)}.$$
\end{thm}

Let us remark that if $k=n$, then $\mathcal{C}=\mathbb{R}^n$ and the theorem reduces to the main result of \cite{Fi}, the only difference being that here we do not attempt to find any explicit upper bound on the constant $C(n,K)$. However, since all of our arguments are ``constructive," it is possible to find explicit upper bounds on $C(n,K)$ in terms on $n$ and the geometry of $\mathcal{C}$ (see also Section 1.4).          

\subsection{The anisotropic isoperimetric inequality}

As we will show below, our result is strictly related to the quantitative version of the anisotropic isoperimetric inequality proved in \cite{Fi}. To show this link, we first introduce some more notation.   
Suppose $K$ is an open, bounded, convex set, and let \footnote{Usually in the definition of $||\cdot||_{K^*}$, $K$ is assumed to contain the origin. However, this is not needed (see Lemma \ref{6}).}
\begin{equation} \label{2} 
||\nu||_{K*} := \sup\{\nu \cdot z: z\in K\}. 
\end{equation}
The anisotropic perimeter of a set $E$ of finite perimeter (i.e. $P(E|\mathbb{R}^n) < \infty$) is defined as
\begin{equation} \label{4}
P_K(E):=\int_{\mathcal{F} E} ||\nu_E(x)||_{K*} dH^{n-1}(x),
\end{equation}
where $\mathcal{F}E$ is the reduced boundary of $E$, and $\nu_E: \mathcal{F}E \rightarrow \mathbb{S}^{n-1}$ is the measure theoretic \textit{outer} unit normal (see Section 2). Note that for $\lambda>0$, $P_K(\lambda E)=\lambda^{n-1} P_K(E)$ and $P_K(E)=P_K(E+x_0)$ for all $x_0 \in \mathbb{R}^n$. If $E$ has a smooth boundary, $\mathcal{F} E = \partial E$ so that for $K=B_1$ we have $P_K(E)=H^{n-1}(\partial E).$ In general, one can think of $|| \cdot ||_{K*}$ as a weight function on unit vectors. Indeed, $P_K$ has been used to model surface tensions in the study of equilibrium configurations of solid crystals with sufficiently small grains (see e.g. \cite{Wu}, \cite{He}, \cite{Ty}) and also in modeling surface energies in phase transitions (see \cite{Gu}). 

The anisotropic isoperimetric inequality states 
\begin{equation} \label{b2}
n|K|^{\frac{1}{n}}|E|^{\frac{n-1}{n}} \leq P_K(E).
\end{equation}             
This estimate (including equality cases) is well known in the literature (see e.g. \cite{Di, VS, DP, Ty, FM, BM, MS}). In particular, Gromov \cite{MS} uses certain properties of the Knothe map from $E$ to $K$ in order to establish (\ref{b2}). However, as pointed out in \cite{Cord} and \cite{Fi}, the argument may be repeated verbatim if one uses the Brenier map instead. This approach leads to certain estimates which are helpful in proving a sharp stability theorem for (\ref{b2}) (see \cite[Theorem 1.1]{Fi}). Using the anisotropic perimeter, we now introduce the isoperimetric deficit of $E$ 
\begin{equation} \label{5}
\delta_K(E):= \frac{P_K(E)}{n|K|^{\frac{1}{n}}|E|^{\frac{n-1}{n}}} -1.
\end{equation}  
 
Note that $\delta_K(\lambda E)=\delta_K(E)$ and $\delta_K(E+ x_0)=\delta_K(E)$ for all $\lambda>0$ and $x_0 \in \mathbb{R}^n$. Thanks to (\ref{b2}) and the associated equality cases, we have $\delta_K(E) \geq 0$ with equality if and only if $E$ is equal to $K$ (up to a scaling and translation). Note also the similarity between $\mu$ and $\delta_K$. Indeed, they are both scaling invariant; however, $\mu$ may not be translation invariant (depending on $\mathcal{C}$). We denote the asymmetry index of $E$ by               
\begin{equation}
A(E):= \displaystyle \inf \bigg\{ \frac{E \Delta (x_0+rK)}{|E|}: x_0 \in \mathbb{R}^n, |rK|=|E| \bigg \}.
\end{equation}      
The general stability problem consists of proving an estimate of the form
\begin{equation} \label{b3}
A(E) \leq C \delta_K(E)^{\frac{1}{\beta}},
\end{equation}  
where $C=C(n,K)$ and $\beta=\beta(n,K)$. In the Euclidean case (i.e. $K=B_1$), Hall conjectured that (\ref{b3}) should hold with $\beta=2$, and this was confirmed by Fusco, Maggi, Pratelli \cite{FMP}. Indeed, the $\frac{1}{2}$ exponent is sharp (see e.g. \cite[Figure 4]{Ma}). Their proof depends heavily on the full symmetry of the Euclidean ball. For the general case when $K$ is a generic convex set, non-sharp results were obtained by Esposito, Fusco, Trombetti \cite{EFT}, while the sharp estimate was recently obtained by Figalli, Maggi, Pratelli \cite{Fi}. Their proof uses a technique based on optimal transport theory. For more information about the history of (\ref{b3}), we refer the reader to \cite{Fi} and \cite{Ma}. 

\subsection{Sketch of the proof of Theorem \ref{t1}}

We now provide a short sketch of the proof of Theorem \ref{t1} for the case when
$|E|=|K|$ and $E$ has a smooth boundary.
The first key observation is that the relative isoperimetric inequality inside a convex cone is a direct consequence of the anisotropic isoperimetric inequality with $K=B_1 \cap \mathcal{C}$. Indeed, as follows from the argument in Section \ref{sect:isop ineq}, $P_K(E) \leq H^{n-1}(\partial E \cap \mathcal C)$, so \eqref{i1} follows immediately from \eqref{b2}.
This observation suggests that one may exploit Gromov's argument in a similar way as in the proof of \cite[Theorem 1.1]{Fi} to obtain additional information
on $E$.
Indeed, we can show that there exists a vector $\alpha=\alpha(E) \in \mathbb{R}^n$ such that \footnote{The existence of a vector $\alpha$ such that (\ref{12r}) holds is exactly the main result in \cite{Fi}. However, here we need to show that we can find a vector such that both (\ref{12o}) and (\ref{12r}) hold simultaneously.}  
\begin{equation} \label{12o} 
\int_{\partial E \cap \mathcal{C}} |1-|x-\alpha|| dH^{n-1} \leq C(n,K) \sqrt{\delta_K(E)}, 
\end{equation}
\begin{equation} \label{12r}
|E \Delta (\alpha + K)| \leq C(n,K)\sqrt{\delta_K(E)}.
\end{equation}
Let us write $\alpha=(\alpha_1, \alpha_2)$, with $\alpha_1 \in \mathbb{R}^{k}$ and $\alpha_2 \in \mathbb{R}^{n-k}$. Moreover, let $\tilde{E}:= E-(\alpha_1,0)$. Then using that $\mathcal{C}=\mathbb{R}^k \times \mathcal{\tilde{C}}$, we obtain $\partial E \cap \mathcal{C} - (\alpha_1,0)=\partial \tilde{E} \cap \mathcal{C}$; therefore,    
\begin{equation} \label{12a}
\int_{\partial \tilde{E} \cap \mathcal{C}} |1-|x-(0,\alpha_2)|| dH^{n-1} \leq C(n,K) \sqrt{\delta_K(E)}, 
\end{equation}
\begin{equation} \label{12e}
|\tilde{E} \Delta ((0,\alpha_2) + K)| \leq C(n,K)\sqrt{\delta_K(E)}.
\end{equation}
Since $\delta_K(E) \leq \mu(E)$ (see Corollary \ref{8}), (\ref{12a}) and (\ref{12e}) hold with $\mu(E)$ in place of $\delta_K(E)$ (see Lemmas \ref{15} \& \ref{l0}). Thanks to (\ref{12e}), we see that our result would readily follow if we can show 
\begin{equation} \label{t16}
|\alpha_2| \leq C(n,K)\sqrt{\mu(E)}.
\end{equation}
Indeed, since $|((0,\alpha_2)+K) \Delta K| \sim |\alpha_2|$ (see Lemmas \ref{17} \& \ref{18}),
\begin{align*}
\displaystyle \frac{|\tilde{E} \Delta K|}{|E|} &\leq \frac{1}{|K|} \bigl(|\tilde{E} \Delta ((0,\alpha_2)+K)| + |((0,\alpha_2)+K) \Delta K| \bigr) \\
&\leq \tilde{C}(n,K)\sqrt{\mu(E)},
\end{align*}
which, of course, implies Theorem \ref{t1}.    
Therefore, we are left with proving (\ref{t16}). Firstly, assume that $\mu(E)$ and $|\alpha_2|$ are sufficiently small (i.e. smaller than a constant depending only on $n$ and $K$). By (\ref{12a}) and the fact that (see Section $2$) $$H^{n-1}(\partial \tilde{E} \cap (B_{\frac{3}{4}}((0,\alpha_2)) \cap \mathcal{C}) ) = P(\tilde{E}| B_{\frac{3}{4}}((0,\alpha_2)) \cap \mathcal{C}),$$ we have          
\begin{align*}
C(n,K) \sqrt{\mu(E)} &\geq  \int_{\partial \tilde{E} \cap \mathcal{C}} |1-|x-(0,\alpha_2)|| dH^{n-1}(x)\\
&\geq \int_{\partial \tilde{E} \cap \mathcal{C} \cap \{|1-|x-(0,\alpha_2)|| \geq \frac{1}{4} \}} |1-|x-(0,\alpha_2)|| dH^{n-1}(x) \\
&\geq \frac{1}{4} H^{n-1}\big(\partial \tilde{E} \cap \mathcal{C} \cap \big\{|1-|x-(0,\alpha_2)|| \geq \frac{1}{4} \big\}\big)\\
&\geq \frac{1}{4}H^{n-1}(\partial \tilde{E} \cap (B_{\frac{3}{4}}((0,\alpha_2)) \cap \mathcal{C}) )\\
&= \frac{1}{4} P(\tilde{E}| B_{\frac{3}{4}}((0,\alpha_2)) \cap \mathcal{C}). 
\end{align*} 
But since $|\alpha_2|$ is small, $B_{\frac{1}{2}}(0) \cap \mathcal{C} \subset B_{\frac{3}{4}}((0,\alpha_2)) \cap \mathcal{C}$; hence, $$P(\tilde{E}| B_{\frac{3}{4}}((0,\alpha_2)) \cap \mathcal{C}) \geq P(\tilde{E}|B_{\frac{1}{2}}(0) \cap \mathcal{C}).$$ Moreover, thanks to the relative isoperimetric inequality inside $B_{\frac{1}{2}}(0) \cap \mathcal{C}$ (see e.g. \cite[Inequality (3.43)]{AFP}), we have that for $\mu(E)$ small enough,   
\begin{align} 
C(n,K) &\sqrt{\mu(E)} \nonumber\\
&\geq \frac{1}{4} c(n, K) \min \big\{|\tilde{E} \cap (B_{\frac{1}{2}}(0) \cap \mathcal{C})|^{\frac{n-1}{n}}, |(B_{\frac{1}{2}}(0) \cap \mathcal{C}) \setminus \tilde{E}|^{\frac{n-1}{n}} \big\} \nonumber \\
&\geq \frac{1}{4} c(n, K) \min\big \{|\tilde{E} \cap (B_{\frac{1}{2}}(0) \cap \mathcal{C})|, |(B_{\frac{1}{2}}(0) \cap \mathcal{C}) \setminus \tilde{E}| \big\} \nonumber \\ 
&=\frac{1}{4} c(n, K) |(B_{\frac{1}{2}}(0) \cap \mathcal{C}) \setminus \tilde{E}|, \label{t14}
\end{align}
where in the last step we used that $\tilde{E}$ is close to $(0,\alpha_2)+K$ (see (\ref{12e})) and $|\alpha_2|$ is small. Therefore, using (\ref{12e}) and (\ref{t14}), 
\begin{align}
|(B_{\frac{1}{2}}(0) \cap \mathcal{C}) \setminus ((0,\alpha_2) + K)| &\leq |(B_{\frac{1}{2}}(0) \cap \mathcal{C}) \setminus \tilde{E}| + |\tilde{E} \setminus ((0,\alpha_2) + K)| \nonumber \\
&\leq \frac{4C(n,K)}{c(n,K)} \sqrt{\mu(E)} + C(n,K) \sqrt{\mu(E)} \nonumber \\
&\leq \tilde{C}(n,K) \sqrt{\mu(E)} \label{t17}. 
\end{align}
Since $\mathcal{\tilde{C}}$ contains no lines, by some simple geometric considerations one may reduce the problem to the case when $\alpha_2 \in \{(x_{k+1},\ldots, x_n) \in \mathbb{R}^{n-k}: x_n \geq 0 \}$ (see Lemma \ref{19}), and then it is not difficult to prove $$c(n,K)|\alpha_2| \leq |(B_{\frac{1}{2}}(0) \cap \mathcal{C}) \setminus ((0, \alpha_2) + K)|,$$ which combined with (\ref{t17}) establishes (\ref{t16}), and hence, the theorem.

Next, we briefly discuss the assumptions for the sketch of the proof above. Indeed, one may remove the size assumption on $|\alpha_2|$ by showing that if $\mu(E)$ is small enough, then $|\alpha_2|$ will be automatically small
(see Proposition \ref{29}).\footnote{Let us point out that this is a nontrivial fact. Indeed, in our case we want to prove in an explicit, quantitative way that $\mu(E)$ controls $\alpha_2(E)$; hence, we want to avoid any compactness argument. However, even using compactness, we do not know any simple argument which shows that $\alpha_2(E) \rightarrow 0$ as $\mu(E) \rightarrow 0$.}   Furthermore, we may freely assume that $\mu(E)$ is small    
since if $\mu(E) \geq c(n,K)>0$, then the theorem is trivial: 
\begin{equation*}
\frac{|E \Delta (sK)|}{|E|} \leq 2 \leq \frac{2}{\sqrt{c(n,K)}} \sqrt{\mu(E)}.
\end{equation*}
The regularity of $E$ was used in order to apply the Sobolev-Poincar{\'e} type estimate \cite[Lemma 3.1]{Fi} which yields (\ref{12a}) (see Lemma \ref{15}). If $E$ is a general set of finite perimeter in $\mathcal{C}$ with finite mass and small relative deficit, then Lemma \ref{21a} tells us that $E$ has a sufficiently regular subset $G$ so that $|E\setminus G|$ and $\mu(G)$ are controlled by $\mu(E)$. Combining this fact with the argument above yields the theorem for general sets of finite perimeter (see Proposition \ref{21}). Lastly, the assumption on the mass of $E$ (i.e. $|E|=|K|$) can be removed by a simple scaling argument. 

\subsection{Sharpness of the result}
We now discuss the sharpness of the estimate in Theorem \ref{t1}. Indeed, it is well known that there exists a sequence of ellipsoids $\{E_h\}_{h \in \mathbb{N}}$, symmetric with respect to the origin and converging to the ball $B_1$, such that $$\displaystyle \lim_{h\rightarrow \infty} \sup \frac{\sqrt{\delta_{B_1}(E_h)}}{|E_h \Delta (s_h B_1)|} < \infty, \hskip .3in \displaystyle \lim_{h\rightarrow \infty} \delta_{K_0}(E_h) = 0,$$ where $s_h=\bigg(\frac{|E_h|}{|B_1|}\bigg)^{\frac{1}{n}}$ (see e.g. \cite[pg. 382]{Ma}). Consider the cone $\mathcal{C}=\{x \in \mathbb{R}^n : x_1, \ldots, x_n>0\}$ and set $\tilde{E}_h:=E_h \cap \mathcal{C}$. By symmetry, it follows that $\delta_{B_1}(\tilde{E}_h) = \frac{1}{2^n}\delta_{B_1}(E_h)$ and $|\tilde{E}_h \Delta (s_h K)|=\frac{1}{2^n} |E_h \Delta (s_h B_1)|$. We also note that $$P(\tilde{E}_h | \mathcal{C})=H^{n-1}(\partial \tilde{E}_h \cap \mathcal{C})=\frac{1}{2^n}H^{n-1}(\partial E_h)=\frac{1}{2^n} P_{B_1}(E_h),$$ $$|\tilde{E}_h|=\frac{1}{2^n} |E_h|, \hskip .1in |B_1 \cap \mathcal{C}|=\frac{1}{2^n} |B_1|.$$ Therefore,

\begin{align*}
\mu(\tilde{E}_h)&=\frac{P(\tilde{E}_h | \mathcal{C})}{n|B_1 \cap \mathcal{C}|^{\frac{1}{n}}|\tilde{E}_h|^{\frac{n-1}{n}}} -1 =\frac{\frac{1}{2^n}H^{n-1}(\partial E_h)}{n (\frac{1}{2^n} |B_1|)^{\frac{1}{n}}(\frac{1}{2^n}|E|)^{\frac{n-1}{n}}} -1 \\
&=\frac{P_{B_1}(E_h)}{n|B_1 \cap \mathcal{C}|^{\frac{1}{n}}|E_h|^{\frac{n-1}{n}}} -1=\delta_{B_1}(E_h), 
\end{align*}
and we have $$\displaystyle \lim_{h \rightarrow \infty} \sup \frac{\sqrt{\mu(\tilde{E}_h)}}{|\tilde{E}_h \Delta (s_h K)|} < \infty.$$ This example shows that the $\frac{1}{2}$ exponent in the theorem cannot, in general, be replaced by something larger.

One may wonder whether it is possible for Theorem \ref{t1} to hold with a constant depending only on the dimension and not on the cone. Indeed, in \cite[Theorem 1.1]{Fi}, the constant does not depend on the convex set associated to the anisotropic perimeter. However, this is not so in our case. To see this, consider a sequence of open, symmetric cones in $\mathbb{R}^2$ indexed by their opening $\theta$. Let $E_\theta$ be a unit half-ball along the boundary of the cone $\mathcal{C}_\theta$ disjoint from $s_{\theta}K_\theta$ (see Figure 1), where $s_{\theta}=\bigg(\frac{|E_{\theta}|}{|B_1 \cap \mathcal{C}_{\theta}|}\bigg)^{\frac{1}{2}}$. Note that $\mu(E_\theta)=\frac{\pi}{2\big(\frac{\theta}{2}\big)^{\frac{1}{2}}\big(\frac{\pi}{2}\big)^{\frac{1}{2}}}-1$. Therefore, $$\displaystyle \lim_{\theta \rightarrow \pi^-} \frac{|E_\theta \Delta (s_{\theta}K_\theta)|}{|E_\theta| \sqrt{\mu(E_\theta)}}= \displaystyle \lim_{\theta \rightarrow \pi^-} \frac{2}{\sqrt{\frac{\pi}{\sqrt{\theta \pi}}-1}} = \infty.$$ 
\begin{figure}[h!]
\centering 
\includegraphics[width=.5 \textwidth]{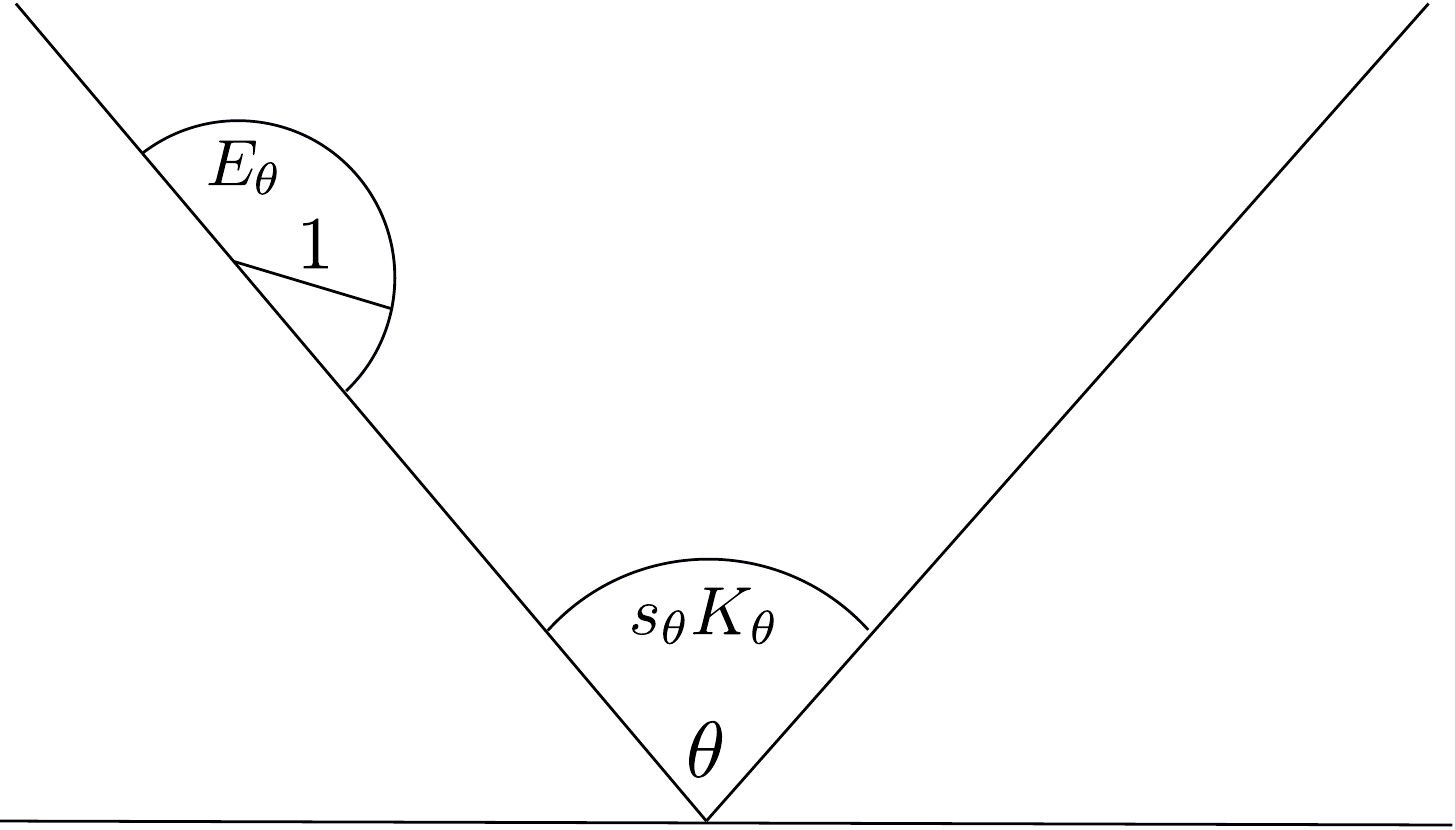}
\caption{An example which shows that the constant in Theorem \ref{t1} cannot be replaced by a constant depending only on the dimension.}
\end{figure}

\section{Preliminaries}
\subsection{Initial setup}

Endow the space $\mathbb{R}^{n \times n}$ of $n \times n$ tensors with the metric $|A|=\sqrt{\operatorname{trace}(A^TA)}$, where $A^T$ denotes the transpose of $A$. Let $T \in L_{\operatorname{loc}}^1(\mathbb{R}^n; \mathbb{R}^n)$ and denote the distributional derivative of $T$ by $DT$, i.e. $DT$ is an $n\times n$ matrix of measures $D_jT^i$ in $\mathbb{R}^n$ satisfying
\begin{equation*}
\int_{\mathbb{R}^n} T^i \frac{\partial \phi}{\partial x_j} dx = - \int_{\mathbb{R}^n} \phi \hskip.05in d D_jT^i \hskip .1in \forall \phi \in C^1_c(\mathbb{R}^n), \hskip .1in i,j=1,\ldots,n.
\end{equation*}
If $C\subset \mathbb{R}^n$ is a Borel set, then
$DT(C)$ is the $n \times n$ tensor whose entries are given by
$\bigl(D_jT^i(C) \bigr)_{i,j=1,\ldots,n}$,
and $|DT|(C)$ is the total variation of $DT$ on $C$ with respect to the metric defined above, i.e.  
\begin{equation*}
|DT|(C)=\displaystyle \sup \bigg\{ \displaystyle \sum_{h\in \mathbb{N}} \sum_{ij} |D_iT^j(C_h)| : C_h \cap C_k = \emptyset, \displaystyle \bigcup_{h \in \mathbb{N}} C_h \subset C \bigg \}.
\end{equation*}
Let $BV(\mathbb{R}^n; \mathbb{R}^n)$ be the set of all $T \in L^1(\mathbb{R}^n; \mathbb{R}^n)$ with $|DT|(\mathbb{R}^n) < \infty$. For such a $T$, decompose $DT=\nabla T dx + D_sT$, where $\nabla T$ is the density with respect to the Lebesgue measure and $D_sT$ is the corresponding singular part. Denote the distributional divergence of $T$ by $\operatorname{Div}T := \operatorname{trace}(DT)$, and let $\operatorname{div}(T):= \operatorname{trace}(\nabla T(x))$. Then we have $\operatorname{Div}T= \operatorname{div}T dx + \operatorname{trace}(D_sT)$. If $DT$ is symmetric and positive definite, note that 
\begin{equation} \label{las1}
\operatorname{trace}(D_sT) \geq 0.
\end{equation}

If $E$ is a set of finite perimeter in $\mathbb{R}^n$, then the reduced boundary $\mathcal{F} E$ of $E$ consists of all points $x\in \mathbb{R}^n$ such that $0<|D \mathbf{1}_E|(B_r(x)) < \infty$ for all $r>0$ and the following limit exists and belongs to $\mathbb{S}^{n-1}$: $$\displaystyle \lim_{r\rightarrow 0^{+}} \frac{D \mathbf{1}_E(B_r(x))}{|D \mathbf{1}_E|(B_r(x))}=:-\nu_E(x).$$ We call $\nu_E$ the measure theoretic \textit{outer} unit normal to $E$. By the well-known representation of the perimeter in terms of the Hausdorff measure, one has $P(E|\mathcal{C})=H^{n-1}(\mathcal{F} E \cap \mathcal{C})$ (see e.g. \cite[Theorem 3.61]{AFP} and \cite[Equation (3.62)]{AFP}). This fact along with one of the equality cases in (\ref{i1}) -- $n|B_1\cap \mathcal{C}|=H^{n-1}(\partial B_1 \cap \mathcal{C})$ -- yields the following useful representation of the relative deficit (recall that $s>0$ satisfies $|E|=|sK|$): 
\begin{equation} \label{fin20}
\mu(E)=\frac{H^{n-1}(\mathcal{F} E \cap \mathcal{C}) - H^{n-1}(\partial B_s \cap \mathcal{C})}{H^{n-1}(\partial B_s \cap \mathcal{C})}.
\end{equation}

Next, if $T \in BV(\mathbb{R}^n; \mathbb{R}^n)$, then for $H^{n-1}$- a.e. $x \in \mathcal{F} E$ there exists an inner trace vector $\operatorname{tr}_E(T)(x) \in \mathbb{R}^n$ (see \cite[Theorem 3.77]{AFP}) which satisfies $$\displaystyle \lim_{r \rightarrow 0^+} \frac{1}{r^n} \int_{B_r(x) \cap \{y: (y-x) \cdot \nu_E(x)<0 \}} |T(y)-\operatorname{tr}_E(T)(x)|dy=0.$$ 
Furthermore, $E^{(1)}$ denotes the set of points in $\mathbb{R}^n$ having density $1$ with respect to $E$; i.e. $x \in E^{(1)}$ means $$\displaystyle \lim_{r\rightarrow 0^{+}} \frac{|E \cap B_r(x)|}{|B_r(x)|}=1.$$ 

Having developed the necessary notation, we are ready to state the following general version of the divergence theorem (see e.g. \cite[Theorem 3.84]{AFP}) which will help us prove the isoperimetric inequality for convex cones (i.e. Theorem \ref{7}): 
\begin{equation} \label{7b}
\operatorname{Div}T (E^{(1)}) = \int_{\mathcal{F} E} \operatorname{tr}_E(T)(x) \cdot \nu_E(x) dH^{n-1}(x).  
\end{equation}

Now we develop a few more tools that will be used throughout the paper. Fix $K:=B_1 \cap \mathcal{C}$, and let $$\mathcal{D}:=\{ E \subset \mathcal{C}: P(E|\mathcal{C}) < \infty, |E|<\infty \}.$$ To apply the techniques in \cite{Fi}, we need a convex set that contains the origin. Therefore, let us translate $K$ by the vector $x_0 \in -K$ which minimizes the ratio $\frac{M_{K_{0}}}{m_{K_{0}}}$, where $K_0=K+x_0$, 
\begin{equation} \label{9}
m_{K_0} := \inf \{||\nu||_{K_0*} : \nu \in \mathbb S^{n-1} \}>0, \hskip .1in M_{K_0} := \sup \{||\nu||_{K_0*} : \nu \in \mathbb S^{n-1} \}>0,
\end{equation}
and $||\nu||_{K_0*}$ is defined as in (\ref{2}). Next, we introduce the \textit{Minkowski gauge} associated to the convex set $K_0$:
\begin{equation} \label{1}
||z||_{K_0}:= \inf \bigg\{\lambda>0 : \frac{z}{\lambda} \in K_0 \bigg\}.
\end{equation}
Note that the convexity of $K_0$ implies the triangle inequality for $||\cdot||_{K_0}$ so that it behaves sort of like a norm; however,  if $K_0$ is not symmetric with respect to the origin, $||x||_{K_0} \neq ||-x||_{K_0}$. Hence, this ``norm" is in general not a true norm. Nevertheless, the following estimates relate this quantity with the standard Euclidean norm $|\cdot|$ (see \cite[Equations (3.2) and (3.9)]{Fi}):  
\begin{equation} \label{10}
\frac{|x|}{M_{K_0}} \leq ||x||_{K_0} \leq \frac{|x|}{m_{K_0}},
\end{equation}

\begin{equation} \label{10a}
||y||_{K_0*} \leq \frac{M_{K_0}}{m_{K_0}} ||-y||_{K_0*}.
\end{equation}

Recall that the isoperimetric deficit $\delta_K(\cdot)$ is scaling and translation invariant in its argument. The next lemma states that it is also translation invariant in $K$ (observe that if $z_0+K$ does not contain the origin, then $||\cdot||_{z_0+K}$ can also be negative in some direction).
\begin{lem} \label{6}
Let $E \in \mathcal{D}$. Then $\delta_{z_{0}+K}(E)=\delta_{K}(E)$ for all $z_0 \in \mathbb{R}^n$. 
\end{lem}

\begin{proof}
It suffices to prove $P_{z_0+K}(E)=P_K(E)$. 

\begin{align*}
P_{z_0+K}(E) &= \int_{\mathcal{F} E}  \sup\{\nu_E(x) \cdot z: z \in z_0 + K\}  dH^{n-1}(x)\\
&= \int_{\mathcal{F} E}  \sup\{\nu_E(x) \cdot (z_0+z): z \in K\}  dH^{n-1}(x) \\
&= \int_{\mathcal{F} E} (\nu_E(x) \cdot z_0 + \sup\{\nu_E(x) \cdot z: z \in K\})  dH^{n-1}(x) \\
&= \int_{\mathcal{F} E} \nu_E(x) \cdot z_0 dH^{n-1}(x) + P_K(E).
\end{align*}
By using the divergence theorem for sets of finite perimeter \cite[Equation (3.47)]{AFP}, we obtain 
$$\int_{\mathcal{F} E} \nu_E(x) \cdot z_0 dH^{n-1}(x) = \int_{E} \operatorname{div}(z_0) dx=0,$$ 
which proves the result.  
\end{proof}

\subsection{Isoperimetric inequality inside a convex cone}
\label{sect:isop ineq}

Here we show how to use Gromov's argument to prove the relative isoperimetric inequality for convex cones. As discussed in the introduction, the first general proof of the inequality was due to Lions and Pacella \cite{LP} and is based on the Brunn-Minkowski inequality. The equality cases were considered in \cite{LP} for the special case when $\mathcal{C} \setminus \{0\}$ is smooth. Our proof of the inequality closely follows the proof of \cite[Theorem $2.3$]{Fi} with some minor modifications.
\begin{thm} \label{7} 
Let $\mathcal{C}$ be an open, convex cone and $|E|<\infty$. Then 
\begin{equation} \label{revised}
n|E|^{\frac{n-1}{n}} |K|^{\frac{1}{n}}\leq H^{n-1}(\mathcal{F} E \cap \mathcal{C}). 
\end{equation} \label{7aaa}
Moreover, if $\mathcal{C}$ contains no lines, then equality holds if and only if $E = sK$. 
\end{thm}

\begin{proof} \

\noindent \textbf{Proof of (\ref{revised}).} By rescaling, if necessary, we may assume that $|K| = |E|$ (i.e. $s=1$). Define the probability densities $d\mu^{+}(x)= \frac{1}{|E|}\mathbf{1}_{E}(x) dx$ and $d\mu_{-}(y) = \frac{1}{|K|} \mathbf{1}_K (y) dy$. By classical results in optimal transport theory, it is well known that there exists an a.e. unique map $T : E \rightarrow K$ (which we call the Brenier map) such that $T=\nabla \phi$ where $\phi$ is convex, $T \in BV(\mathbb{R}^n; K),$ and $\det(\nabla T(x))=1$ for a.e. $x \in E$ (see e.g. \cite{Br, McC1, McC2, AA}). Moreover, since $T$ is the gradient of a convex function with positive Jacobian, $\nabla T(x)$ is symmetric and nonnegative definite; hence, its eigenvalues $\lambda_k(x)$ are nonnegative for a.e. $x \in \mathbb{R}^n$. As a result, we may apply the arithmetic-geometric mean inequality to conclude that for a.e. $x \in E$, 
\begin{equation} \label{7a}
n=n\bigl(\det \nabla T(x) \bigr)^{\frac{1}{n}} = n \biggl(\prod_{k=1}^n \lambda_k(x)\biggr)^{\frac{1}{n}} \leq \displaystyle \sum_{k=1}^n \lambda_k(x) = \operatorname{div}T(x).
\end{equation}
Therefore, 
\begin{align}
n |E|^{\frac{n-1}{n}} |K|^{\frac{1}{n}} &= n|E| = n \int_{E} \det(\nabla T(x))^{\frac{1}{n}} dx\notag \\
&\leq \int_{E} \operatorname{div}T(x) dx = \int_{E^{(1)}} \operatorname{div}T(x) dx, \label{7c bis}
\end{align}
where we recall that $E^{(1)}$ denotes the set of points with density $1$ (see Section 2.1).
Next, we use (\ref{las1}) and (\ref{7b}):
\begin{align}
\int_{E^{(1)}} \operatorname{div}T(x) dx &\leq \int_{E^{(1)}} \operatorname{div}T(x) dx + (\operatorname{Div} T)_s(E^{(1)}) \notag \\
&=\operatorname{Div} T (E^{(1)})= \int_{\mathcal{F} E} \operatorname{tr}_E(T)(x) \cdot \nu_E(x) dH^{n-1}(x). \label{7c} 
\end{align} 
By the convexity of $K$ and the fact that $T(x) \in K$ for a.e. $x \in E$, it follows that $\operatorname{tr}_E(T)(x) \in \bar{K}$, so by the definition of $||\cdot||_{K^*}$, $$\operatorname{tr}_E(T)(x) \cdot \nu_E(x) \leq ||\nu_E(x)||_{K*}.$$ Hence,    
\begin{align}
\int_{\mathcal{F} E} \operatorname{tr}_E(T)(x) \cdot \nu_E(x) dH^{n-1}(x) \leq \int_{\mathcal{F} E} ||\nu_E(x)||_{K*}  dH^{n-1}(x) = P_K(E) \label{las2}.   
\end{align}
Furthermore, note that if $z\in K$, then $|z| \leq 1$; therefore, $$||\nu_E(x)||_{K*}= \sup\{\nu_E(x) \cdot z: z\in K\} \leq 1.$$ Moreover, observe that by the definition of $||\cdot||_{K*}$, it follows easily that $||\nu_{\mathcal{C}}(x)||_{K*}=0$ for $H^{n-1}$ -a.e. $x \in \partial \mathcal{C} \setminus \{0\}$; therefore, $||\nu_E(x)||_{K*}=0$ for $H^{n-1}$ -a.e. $x \in \mathcal{F}E \cap \partial \mathcal{C}$. Thus,    
\begin{align*}
\int_{\mathcal{F} E} ||\nu_E(x)||_{K*}  dH^{n-1}(x)&=\int_{\mathcal{F} E \cap \mathcal{C}} ||\nu_E(x)||_{K*}  dH^{n-1}(x) \leq H^{n-1}(\mathcal{F} E \cap \mathcal{C}), 
\end{align*} 
and this proves the inequality. 

\noindent \textbf{Equality case.} If $E=K$, then $T(x)=x$ and it is easy to check that equality holds in each of the inequalities above. Conversely, suppose there is equality. In particular, $n|K|=P_K(E)$. By \cite{FM} (see also \cite[Theorem A.1]{Fi}), we obtain that $E=K+a$ with $a \in \mathcal{\bar{C}}$. Next, we will use the following identity which is valid for any $v \in \bar{\mathcal{C}}$: 
\begin{align*}
P(v+K|\mathcal{C})=P(K|\mathcal{C})+H^{n-1}(S_v),
\end{align*}
where
\begin{align*} 
S_v:&=\{x\in \mathcal{F}\mathcal{C} \cap B_1: \langle \nu_{\mathcal{C}}(x), v\rangle \neq 0\}\\
&=\{x\in \mathcal{F}\mathcal{C} \cap B_1: \langle \nu_{\mathcal{C}}(x), v\rangle < 0\}.
\end{align*}
By the previous equality case, we know $n|K|=P(K|\mathcal{C})$, therefore, 
\begin{align*} 
P(K|\mathcal{C})=P(E|\mathcal{C})=P(a+K|\mathcal{C})=P(K|\mathcal{C})+H^{n-1}(S_a),
\end{align*}
and we obtain $H^{n-1}(S_a)=0$. This implies $\langle \nu_{\mathcal{C}}(x),a\rangle = 0$ for $H^{n-1}$-almost every $x \in \mathcal{F} \mathcal{C}$. Hence, $D \hskip .02 in \mathbf{1}_{\mathcal{C}} =0$ in the direction defined by $a$, which gives $\mathbf{1}_{\mathcal{C}}(x)=\mathbf{1}_{\mathcal{C}}(x+ta)$ for all $t\in \mathbb{R}$. However, by assumption, $\mathcal{C}$ contains no lines; thus, $a=0$ and we conclude the proof.  
\end{proof}    

\begin{cor} \label{8} 
If $E \in \mathcal{D}$, then $\delta_K(E) \leq \mu(E)$.
\end{cor} 

\begin{proof}
Since the inequality is scaling invariant, we may assume that $|E|=|K|$. From (\ref{las2}) and the fact that $n|K|=H^{n-1}(\partial B_1 \cap \mathcal{C})$ we obtain $$P_K(E) - n|K| \leq H^{n-1}(\mathcal{F} E \cap C) - n|K| = H^{n-1}(\mathcal{F} E \cap \mathcal{C}) - H^{n-1}(\partial B_1 \cap \mathcal{C}).$$ Dividing by $n|K|$ and using the representation of $\mu(E)$ given by (\ref{fin20}) yields the result.  
\end{proof}      

\begin{cor} \label{13} 
Let $E \in \mathcal{D}$ with $|E|=|K|$, and let $T_0: E \rightarrow K_0$ be the Brenier map from $E$ to $K_0$. Then $$\int_{\mathcal{F} E \cap \mathcal{C}} \bigl (1-|\operatorname{tr}_E(T_0-x_0)(x)| \bigr ) dH^{n-1}(x) \leq n|K| \mu(E).$$
\end{cor} 

\begin{proof} Let $T:E\rightarrow K$ be the Brenier map from $E$ to $K$ so that $T$ is the a.e. unique gradient of a convex function $\phi$. Then $T_0(x)=T(x)+x_0$ (this follows easily from the fact that $T(x)+x_0=\nabla \phi(x)+x_0=\nabla (\phi(x)+x_0 \cdot x)$ and $\phi(x)+x_0 \cdot x$ is still convex). Therefore, by \eqref{7c bis} and $(\ref{7c})$,
\begin{equation} \label{13a}
n|E| \leq \int_{\mathcal{F} E} \operatorname{tr}_E(T_0-x_0)(x) \cdot \nu_E(x) dH^{n-1}(x) .
\end{equation}
Next, we recall from the proof of Theorem \ref{7} that $\operatorname{tr}_E(T_0-x_0)(x) \in \bar{K}$. Hence, $\operatorname{tr}_E(T_0-x_0)(x) \cdot \nu_E(x) \leq 0$ for $H^{n-1}$ a.e. $x \in \mathcal{F} E \cap \partial \mathcal{C}$ and $|\operatorname{tr}_E(T_0-x_0)(x)|\leq1$ for $H^{n-1}$ a.e. $x \in \mathcal{F} E \cap \mathcal{C}$. Therefore, using (\ref{13a}),  
\begin{align*}
n|E| &\leq \int_{\mathcal{F} E \cap \mathcal{C}} \operatorname{tr}_E(T_0-x_0)(x) \cdot \nu_E(x) dH^{n-1}(x) \\
&\leq \int_{\mathcal{F} E \cap \mathcal{C}} |\operatorname{tr}_E(T_0-x_0)(x)| dH^{n-1}(x) \leq H^{n-1}(\mathcal{F} E \cap \mathcal{C}).
\end{align*}
\noindent The fact that $n|E|=n|K|=H^{n-1}(\partial B_1 \cap \mathcal{C})$ finishes the proof.
\end{proof}

\section{Proof of Theorem \ref{t1}}
We split the proof in several steps. In Section $3.1$, we collect some useful technical tools. Then in Section $3.2$, we prove Theorem \ref{t1} under the additional assumption that $E$ is close to $K$ (up to a translation in the first $k$ coordinates). Finally, we remove this assumption in Section $3.3$ to conclude the proof of the theorem.

Let $\{e_k\}_{k=1}^n$ be the standard orthonormal basis for $\mathbb{R}^n$. Recall that $\mathcal{C}=\mathbb{R}^k \times \mathcal{\tilde{C}}$, where $\tilde{C} \subset \mathbb{R}^{n-k}$ is an open, convex cone containing no lines. Hence, up to a change of coordinates, we may assume without loss generality that $\partial \mathcal{\tilde{C}} \cap \{x_n=0\}=\{0\}.$ With this in mind and a simple compactness argument, we note that 

\begin{equation} \label{19z}
b=b(n,K):= \inf \{t>0: \partial \tilde{B}_{\frac{1}{2}}(0) \cap \mathcal{\tilde{C}} \cap \{x_n<t\} \neq \emptyset \}>0,   
\end{equation}
where $\tilde{B}_{\frac{1}{2}}(0)$ is the ball in $\mathbb{R}^{n-k}$.   
%Indeed, if not, then for a minimizing sequence $t_k$, we may find corresponding $z^k \in \partial B_{\frac{1}{2}} \cap \mathcal{C} \cap \{x_n<t_k\}$. Along a subsequence (still indexed by $k$), we have $z^k \rightarrow z$, with $|z|=\frac{1}{2}$. Denote the $n^{th}$ component of $z^k$ by $z_n^{k}$, so that $0<z_n^k<t_k \rightarrow 0$. Therefore, $z\neq0$, $z \in \partial \mathcal{C}$ and $z_n=0$, a contradiction. Thus, (\ref{19z}) is established. 

Next, we introduce the trace constant of a set of finite perimeter. Recall the definition of $K_0$ given in Section $2.1$, so that (\ref{10}) and (\ref{10a}) hold. Given a set $E \in \mathcal{D}$, let $\tau(E)$ denote the trace constant of $E$, where 

\begin{equation} \label{14}
\tau(E):=\inf \biggl \{\frac{P_{K_0}(F)}{\int_{\mathcal{F} F \cap \mathcal{F} E} ||\nu_E||_{K_0*} dH^{n-1}}: F\subset E, |F| \leq \frac{|E|}{2} \biggr \}.
\end{equation}

\noindent Note that $\tau$ is scaling invariant, and in general $\tau(E) \geq 1$. The trace constant contains valuable information about the geometry of $E$. For example, if $E$ has multiple connected components or outward cusps, then $\tau(E)=1$. In general, sets for which $\tau(E)>1$ enjoy a nontrivial Sobolev-Poincar{\'e} type inequality (see \cite[Lemma 3.1]{Fi}).    
\subsection{Main tools}
In what follows, we list all the technical tools needed in order to prove Theorem \ref{t1}. We decided to move some of the proofs to the appendix in order to make this section more accessible. The following two lemmas are general facts about sets of finite perimeter.

\begin{lem} \label{17} Let $A \subset \mathbb{R}^n$ be a bounded set
of finite perimeter.
Then there exists $C_{\ref{17}}(n,A)>0$ such that for any $y \in \mathbb{R}^{n}$,
$|(y + A) \Delta A| \leq C_{\ref{17}}(n,A) |y|$.
\end{lem}
\begin{lem} \label{18}
Let $A \subset \mathbb{R}^n$ be a bounded set
of finite perimeter.
Then there exist two constants $C_{\ref{18}}(n,A), c_{\ref{18}}(n,A)>0$ such that if $y \in \mathbb{R}^n$, then $$\min\{c_{\ref{18}}(n,A), C_{\ref{18}}(n,A)|y|\} \leq |(y + A) \Delta A|.$$
\end{lem}

\begin{rem} {\rm
Lemma \ref{17} is well-known, and follows 
by applying \cite[Remark 3.25]{AFP} to $u=\mathbf{1}_A$.
Also Lemma \ref{18} should be known, but we have been unable to find a reference. Therefore, we provide a proof in the appendix.}
\end{rem}

\begin{lem} \label{geo}
There exists a bounded, convex set $\tilde{K} \subset B_{\frac{1}{2}}(0)\cap \mathcal{C}$ so that for all $y=(0,\ldots,0,y_{k+1},\ldots,y_n)$ with $y_n \geq 0$, we have 
\begin{equation} 
\tilde{K} \setminus (y+\tilde{K}) = \tilde{K} \setminus (y+\mathcal{C}).
\end{equation} \label{geo1}
Furthermore, if $y_n \leq 0$, then 
\begin{equation} \label{geo2}
 (y+\tilde{K}) \setminus \tilde{K}=(y+\tilde{K}) \setminus \mathcal{C}.
\end{equation}     
\end{lem}

\begin{proof} We will show that one may pick $\tilde{b}=\tilde{b}(n,K)>0$ small enough so that $$\tilde{K}:= B_{\frac{1}{2}}(0)\cap \mathcal{C} \cap \big( \cap_{i=1}^k \{|x_i| < \tilde{b} \} \big) \cap \{x_n < \tilde{b} \}$$ has the desired properties. We will establish (\ref{geo1}) first. Since $y+\tilde{K} \subset y+\mathcal{C}$, it suffices to prove $\tilde{K} \setminus (y+\tilde{K}) \subset \tilde{K} \setminus (y+\mathcal{C})$. If (for contradiction) there exists $x \in \tilde{K}\cap (y+\tilde{K})^c \cap (y+\mathcal{C})$, then $x \in \tilde{K}$ and $x-y \in \mathcal{C} \setminus \tilde{K}$. Since $x \in \tilde{K}$ and $y_n \geq 0$, it follows that $x_n -y_n< \tilde{b}$. Also, $|x_i-y_i|=|x_i| < \tilde{b}$ for $i \in \{1,\ldots,k\}$. Now, $x-y \in \mathcal{C}=\mathbb{R}^k \times \mathcal{\tilde{C}}$, hence, $(x_{k+1}-y_{k+1},\ldots, x_n-y_n) \in \mathcal{\tilde{C}}$. Let $b=b(n,K)$ be the constant from (\ref{19z}), and assume without loss of generality that $\tilde{b}<b$. If $z \in \{x_n=\tilde{b}\} \cap \mathcal{\tilde{C}}$ is such that $|z|=d$, where $d=\sup\{|v|: v\in \mathcal{\tilde{C}}, v_n=\tilde{b}\}$, then $\frac{|z|}{\tilde{b}} = \frac{1/2}{b}$ (see Figure 2). 
\begin{figure}
\centering 
\includegraphics[width=.6 \textwidth]{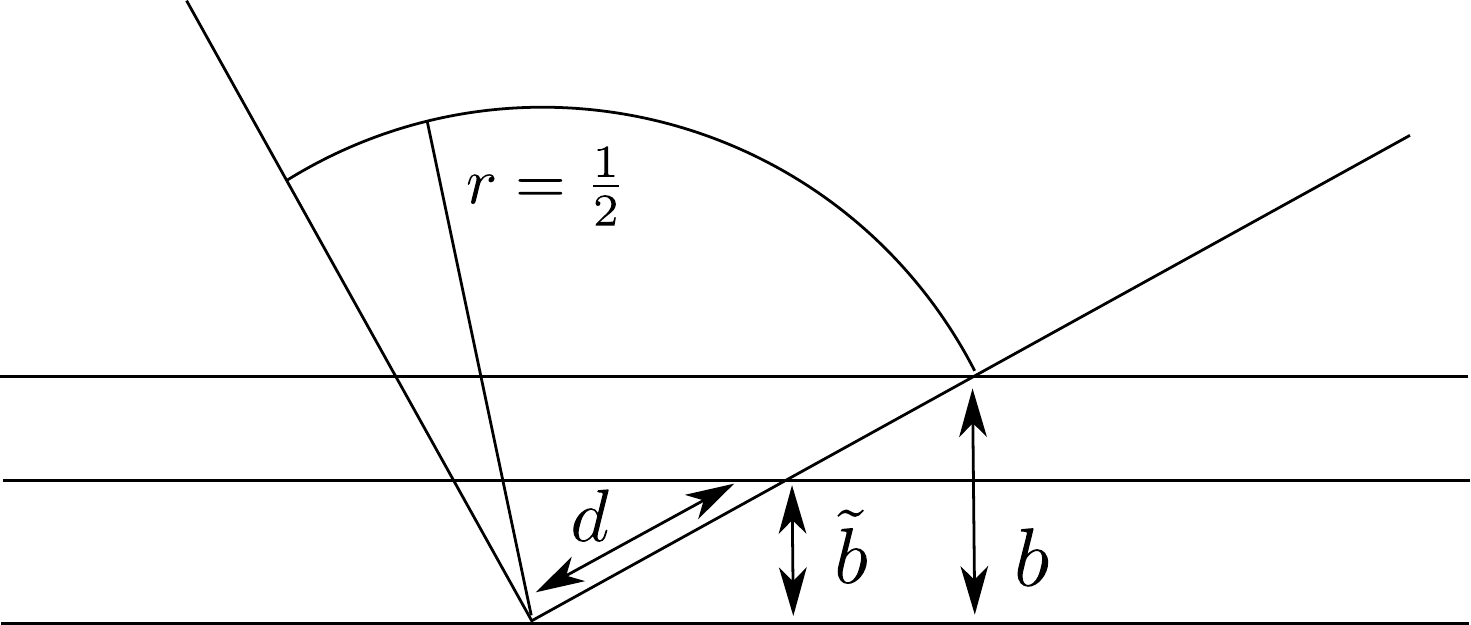}
\caption{$\frac{d}{\tilde{b}} = \frac{1/2}{b}$.}
\end{figure}
Let $\gamma:=\frac{1}{2b}$, $t:=\frac{\tilde{b}}{x_n-y_n}>1$, and recall that $(x_{k+1}-y_{k+1},\ldots, x_n-y_n) \in \mathcal{\tilde{C}}$. Since $\mathcal{\tilde{C}}$ is a cone, we have $w:=t (x_{k+1}-y_{k+1},\ldots, x_n-y_n) \in \mathcal{\tilde{C}}$ with $w_n=\tilde{b}$. Hence, $|w| \leq |z|= \gamma \tilde{b}$, but since $t>1$ we obtain $(x_{k+1}-y_{k+1},\ldots, x_n-y_n) \in \tilde{B}_{\gamma \tilde{b}}(0)$, where $\tilde{B}_{\gamma \tilde{b}}(0)$ denotes the ball in dimension $n-k$. Therefore, $$|x-y|^2 \leq k \tilde{b}^2+(\gamma \tilde{b})^2.$$ Next, pick $M=M(n,K) \in \mathbb{N}$ so that $(k+\gamma^2) \ (\frac{b}{M})^2 < \frac{1}{4}$. Thus, by letting $\tilde{b}:= \frac{b}{M}$, we obtain $x-y \in B_{\frac{1}{2}}(0)$. Therefore, we conclude $x-y \in \tilde{K}$, a contradiction. Hence, (\ref{geo1}) is established. Since $(y+\tilde{K}) \setminus \tilde{K}=y+ \big(\tilde{K} \setminus (-y+\tilde{K})\big)$ and $(y+\tilde{K}) \setminus \mathcal{C}=y+\big(\tilde{K} \setminus (-y+\mathcal{C})\big)$, (\ref{geo2}) follows from (\ref{geo1}).                
\end{proof}
The next lemma (whose proof is postponed to the appendix) tells us that a set with finite mass, perimeter, and small relative deficit has a subset with almost the same mass, good trace constant, and small relative deficit (compare with \cite[Theorem 3.4]{Fi}).
\begin{lem} \label{21a} Let $E \in \mathcal{D}$ with $|E|=|K|.$ Then there exists a set of finite perimeter $G\subset E$ and constants $k(n), c_{\ref{21a}}(n), C_{\ref{21a}}(n,K)>0$ such that if $\mu(E) \leq c_{\ref{21a}}(n),$ then   
\begin{equation} \label{23}
|E \setminus G| \leq \frac{\mu(E)}{k(n)}|E|,
\end{equation}
\begin{equation} \label{24}
\tau(G) \geq 1 + \frac{m_{K_0}}{M_{K_0}}k(n), 
\end{equation}
\begin{equation} \label{24d}
\mu(G) \leq C_{\ref{21a}}(n,K) \mu(E).
\end{equation}
\end{lem}
The big advantage of using $G$ in place of $E$ is that (\ref{24}) implies a nontrivial trace inequality for $G$ which allows us to exploit Gromov's proof in order to prove (\ref{12o}) with $G$ in place of $E$. Indeed, if $E$ is smooth with a uniform Lipschitz bound on $\partial E$, one may take $G=E$. 
\begin{lem} \label{15}
Let $E \in \mathcal{D}$, $|E|=|K|$, and assume $\mu(E) \leq c_{\ref{21a}}(n)$, with $G \subset E$ and $c_{\ref{21a}}(n)$ as in Lemma \ref{21a}. Moreover, let $r>0$ satisfy $|rG|=|K|$. Then there exists $\hat{\alpha}=\hat{\alpha}(E) \in \mathbb{R}^n$ and a constant  $C_{\ref{15}}(n,K)>0$ such that $\int_{\mathcal{F} (rG) \cap \mathcal{C}} |1-|x-\hat{\alpha}|| dH^{n-1} \leq C_{\ref{15}}(n,K) \sqrt{\mu(E)}$.
 \end{lem}
\begin{proof} Let $\tilde{T}_0:rG \rightarrow K_0$ be the Brenier map from $rG$ to $K_0$, and denote by $S_i$ the $i^{th}$ component of $S(x)=\tilde{T}_0(x)-x$. For all $i$, we apply \cite[Lemma $3.1$]{Fi} to the function $S_i$ and the set $rG$ to obtain a vector $a=a(E)=(a_1,\dots,a_n) \in \mathbb{R}^n$ such that

\begin{align*} 
\int_{\mathcal{F} (rG) \cap \mathcal{C}} \operatorname{tr}_{(rG)}(|S_i(x)+a_i)|)||\nu_{(rG)}&||_{K_0*} dH^{n-1}(x)\\ 
&\leq \frac{M_{K_0}}{m_{K_0} (\tau(rG) - 1)} ||-DS_i||_{K_0*}((rG)^{(1)})\\  
&\leq \frac{M_{K_0}^2}{m_{K_0} (\tau(rG) - 1)}|-DS_i|((rG)^{(1)})\\ 
&\leq \frac{M_{K_0}^2}{m_{K_0} (\tau(rG) - 1)}|DS|((rG)^{(1)}),
\end{align*}
where we have used (\ref{10}) in the second inequality. Next, recall that $\tau$ is scaling invariant. Hence, using (\ref{24}) we have   
\begin{equation}
\int_{\mathcal{F} (rG) \cap \mathcal{C}} \operatorname{tr}_{(rG)}(|S_i(x)+a_i)|) ||\nu_{(rG)}||_{K_0*}dH^{n-1}(x) \leq \frac{M_{K_0}^3}{m^2_{K_0}k(n)}|DS|((rG)^{(1)}).
\end{equation}
But by \cite[Corollary $2.4$]{Fi} and Corollary \ref{8}, $$|DS|((rG)^{(1)}) \leq 9n^2 |K|\sqrt{\delta_{K_0}(rG)} \leq 9n^2 |K|\sqrt{\mu(rG)}=9n^2 |K|\sqrt{\mu(G)}.$$ Therefore, by summing over $i=1,2,...,n$ we obtain 
\begin{equation} \label{24aa}
\int_{\mathcal{F} (rG) \cap \mathcal{C}} \operatorname{tr}_{(rG)}(|S(x)+a)|) ||\nu_{(rG)}||_{K_0*} dH^{n-1}(x) \leq \frac{9n^3|K|M_{K_0}^3}{m^2_{K_0}k(n)}\sqrt{\mu(G)}.
\end{equation}   

\noindent Let $\hat{\alpha}=\hat{\alpha}(E):= a+x_0$, with $x_0$ as in the definition of $K_0$ (see (\ref{9})). The triangle inequality implies
\begin{align*}
|1-|x-\hat{\alpha}||&=|1-|x-(a+x_0)|| \leq |1-\operatorname{tr}_{(rG)}(|\tilde{T}_0(x)-x_0|)| \\
& +  |\operatorname{tr}_{(rG)}(\tilde{T}_0(x)-x_0)-(x-(a+x_0))| \\
&=|1-|\operatorname{tr}_{(rG)}(\tilde{T}_0(x)-x_0)|| + \operatorname{tr}_{(rG)}(|\tilde{T}_0(x)-x+a|).
\end{align*}
Hence, by Corollary \ref{13}, (\ref{24aa}), and (\ref{24d}) we have 
\begin{align*}
\int_{\mathcal{F} (rG) \cap \mathcal{C}} m_{K_0} &|1-|x-\hat{\alpha}|| dH^{n-1}(x)\\
&\leq \int_{\mathcal{F}(rG) \cap \mathcal{C}} \big|1-|x-\hat{\alpha}|\big| ||\nu_{(rG)}||_{K_0*} dH^{n-1}(x) \\
&\leq \int_{\mathcal{F} (rG) \cap \mathcal{C}} \big|1-|\operatorname{tr}_{(rG)}(\tilde{T}_0(x)-x_0)|\big| ||\nu_{(rG)}||_{K_0*} dH^{n-1}(x)\\
&+ \int_{\mathcal{F} (rG) \cap \mathcal{C}} \operatorname{tr}_{(rG)}(|S(x))+a|) ||\nu_{(rG)}||_{K_0*} dH^{n-1}(x)\\ 
&\leq M_{K_0}n|K| \mu(G) + 9n^3|K|\frac{M_{K_0}^3}{m^2_{K_0}k(n)}\sqrt{\mu(G)} \\
&\leq M_{K_0}n|K| C_{\ref{21a}}(n,K) \mu(E) + \frac{9n^3|K|M_{K_0}^3}{m^2_{K_0}k(n)}\sqrt{C_{\ref{21a}}(n,K)}\sqrt{\mu(E)}. 
\end{align*}
As $\mu(E) \leq 1$, the result follows.
\end{proof} 
The translation $\hat{\alpha}$ from Lemma \ref{15} can be scaled so that it enjoys some nice properties which we list in the next lemma. The proof is essentially the same as that of \cite[Theorem 1.1]{Fi}, adapted slightly in order to accommodate our setup. However, for the sake of completion, we include it in the appendix. 
 
\begin{lem} \label{l0}
Suppose $E \in \mathcal{D}$ with $|E|=|K|$. Let $\hat{\alpha}=\hat{\alpha}(E)$, $G$, and $r$ be as in Lemma \ref{15}. Define $\alpha=\alpha(E):=\frac{\hat{\alpha}}{r}$. Then there exists a positive constant $C_{\ref{l0}}(n,K)$ such that for $\mu(E) \leq c_{\ref{21a}}(n)$, with $c_{\ref{21a}}(n)$ as in Lemma \ref{21a}, we have 
\begin{equation} \label{16a}
|E \Delta (\alpha+K)| \leq C_{\ref{l0}}(n,K)\sqrt{\mu(E)},   
\end{equation}

\begin{equation} \label{l2}
|(rG) \Delta (\hat{\alpha}+K)| \leq C_{\ref{l0}}(n,K) \sqrt{\mu(E)},  
\end{equation}
and 
\begin{equation} \label{l3}
r \leq 1+\frac{2\mu(E)}{k(n)}.
\end{equation}     
\end{lem}

Next, define $\mathbb{R}_{+}^{n} := \{(x_1, x_2, ..., x_n) \in \mathbb{R}^n: x_n \geq 0 \}$ ($\mathbb{R}_{-}^{n}$ is defined in a similar manner). In the case $\alpha \in \mathbb{R}_{-}^n$, the following lemma tells us that the last $(n-k)$ components of $\alpha$ are controlled by the relative deficit.  

\begin{lem} \label{19}
Let $E \in \mathcal{D}$ with $|E|=|K|$, and let $\alpha=\alpha(E)=(\alpha_1, \alpha_2) \in \mathbb{R}^{k} \times \mathbb{R}^{n-k}$ be as in Lemma \ref{l0}. There exist positive constants $c_{\ref{19}}(n,K), C_{\ref{19}}(n,K)$ such that if $\alpha \in \mathbb{R}_{-}^{n}$ and $\mu(E) \leq c_{\ref{19}}(n,K)$, then $|\alpha_2| \leq C_{\ref{19}}(n,K) \sqrt{\mu(E)}$.
\end{lem}

\begin{proof}
Let $\tilde{K}\subset \mathcal{C}$ be the bounded, convex set given by Lemma \ref{geo}. An application of Lemma $\ref{18}$ and (\ref{geo2}) yields
\begin{align*}
\frac{1}{2} \min\{c_{\ref{18}}(n,\tilde{K}), &C_{\ref{18}}(n,\tilde{K})|\alpha_2| \} \leq \frac{1}{2} |((0, \alpha_2) + \tilde{K}) \Delta \tilde{K}| \\
&=|((0,\alpha_2) + \tilde{K}) \setminus \tilde{K}|=|((0, \alpha_2) + \tilde{K}) \setminus \mathcal{C}|
\end{align*}
Now, note that $E-(\alpha_1,0) \subset \mathcal{C}=\mathbb{R}^k \times \mathcal{\tilde{C}}$; hence, by using this fact and (\ref{16a}) we obtain
\begin{align*}
|((0, \alpha_2) + \tilde{K}) \setminus \mathcal{C}| &\leq |((0,\alpha_2) + \tilde{K}) \setminus (E-(\alpha_1,0))| \leq |(\alpha + K) \setminus E| \\
&\leq C_{\ref{l0}}(n,K) \sqrt{\mu(E)}. 
\end{align*}
Therefore, there exists $c_{\ref{19}}(n,K)>0$ such that for $\mu(E) \leq c_{\ref{19}}(n,K)$, $$ \frac{1}{2}C_{\ref{18}}(n,\tilde{K})|\alpha_2| \leq C_{\ref{l0}}(n,K) \sqrt{\mu(E)}.$$ Thus, the result follows with $C_{\ref{19}}(n,K)= \frac{2C_{\ref{l0}}(n,K)}{C_{\ref{18}}(n,\tilde{K})}$ (note that $K$ completely determines $\tilde{K}$).    
\end{proof}

\subsection{Proof of the result when $|\alpha_2(E)|$ is small}
\begin{prop} \label{21}
Let $E \in \mathcal{D}$ with $|E|=|K|$, and let $\alpha=\alpha(E)=(\alpha_1,\alpha_2) \in \mathbb{R}^k \times \mathbb{R}^{n-k}$ be as in Lemma \ref{l0}. Then there exist positive constants $c_{\ref{21}}(n,K)$, $\tilde{c}_{\ref{21}}(n,K)$, and $C_{\ref{21}}(n,K)$ such that if $\mu(E) \leq c_{\ref{21}}(n,K)$ and $|\alpha_2| \leq \tilde{c}_{\ref{21}}(n,K)$, then $|\alpha_2| \leq C_{\ref{21}}(n,K) \sqrt{\mu(E)}$.   
\end{prop}

\begin{proof} 
Thanks to Lemma $\ref{19}$, we may assume without loss of generality that $\alpha \in \mathbb{R}_{+}^n$. Let $\tilde{G}:= rG-(\hat{\alpha}_1,0)$, with $G$ as in Lemma \ref{21a} and $r>0$ such that $|rG|=|K|$. By Lemma \ref{15} and the fact that $\mathcal{C}=\mathbb{R}^k \times \mathcal{\tilde{C}}$,
\begin{align*}
C_{\ref{15}}(n,K) \sqrt{\mu(E)} &\geq  \int_{\mathcal{F} (rG) \cap \mathcal{C}} |1-|x-\hat{\alpha}|| dH^{n-1}(x)\\
&=  \int_{\mathcal{F} \tilde{G} \cap \mathcal{C}} |1-|x-(0,\hat{\alpha}_2)|| dH^{n-1}(x)\\
&\geq \int_{\mathcal{F} \tilde{G} \cap \mathcal{C} \cap \{|1-|x-(0,\hat{\alpha}_2)|| \geq \frac{1}{4} \}} |1-|x-(0,\hat{\alpha}_2)|| dH^{n-1}(x) \\
&\geq \frac{1}{4} H^{n-1}\big(\mathcal{F} \tilde{G} \cap \mathcal{C} \cap \{|1-|x-(0,\hat{\alpha}_2)|| \geq \frac{1}{4} \}\big)\\
&\geq \frac{1}{4}H^{n-1}(\mathcal{F} \tilde{G} \cap (B_{\frac{3}{4}}((0,\hat{\alpha}_2)) \cap \mathcal{C}) ) = \frac{1}{4} P(\tilde{G}| B_{\frac{3}{4}}((0,\hat{\alpha}_2)) \cap \mathcal{C}).
\end{align*}
However, thanks to (\ref{l3}), $\frac{|\hat{\alpha}_2|}{1+\frac{2}{k(n)}\mu(E)} \leq |\alpha_2|$, so for $|\alpha_2|$ and $\mu(E)$ sufficiently small we have $B_{\frac{1}{2}}(0) \cap \mathcal{C} \subset B_{\frac{3}{4}}((0,\hat{\alpha}_2)) \cap \mathcal{C}$, and this implies $$P(\tilde{G}| B_{\frac{3}{4}}((0, \hat{\alpha}_2)) \cap \mathcal{C})\geq P(\tilde{G}| B_{\frac{1}{2}}(0) \cap \mathcal{C}).$$ Next, by using the relative isoperimetric inequality (apply \cite[Inequality (3.41)]{AFP} to $\mathbf{1}_{(rG)}$ and the set $B_{\frac{1}{2}}(0) \cap \mathcal{C}$), we have that for $\mu(E)$ small enough,  
\begin{align} 
C_{\ref{15}}&(n,K) \sqrt{\mu(E)} \nonumber \\
&\geq \frac{1}{4} c(n,K) \min\big\{|\tilde{G} \cap (B_{\frac{1}{2}}(0) \cap \mathcal{C})|^{\frac{n-1}{n}}, |(B_{\frac{1}{2}}(0) \cap \mathcal{C}) \setminus \tilde{G}|^{\frac{n-1}{n}}\big\} \nonumber\\
&\geq \frac{1}{4} c(n,K) \min\big\{|\tilde{G} \cap (B_{\frac{1}{2}}(0) \cap \mathcal{C})|, |(B_{\frac{1}{2}}(0) \cap \mathcal{C}) \setminus \tilde{G}|\big\}. \label{21e}
\end{align}
Furthermore, 
\begin{align*}
(B_{\frac{1}{2}}(0) \cap \mathcal{C}) \setminus \tilde{G} &\subset K \setminus \tilde{G} \subset \tilde{G} \Delta K\\
&\subset \bigg(\big(rG-(\hat{\alpha}_1,0)\big) \Delta \big(K+(0, \hat{\alpha}_2)\big)\bigg) \cup \bigg(\big(K+ (0, \hat{\alpha}_2)\big) \Delta K\bigg),
\end{align*}
and by using (\ref{l2}), Lemma \ref{17}, and (\ref{l3}), 
\begin{align*}
|(B_{\frac{1}{2}}(0) \cap \mathcal{C}) \setminus \tilde{G}| &\leq C_{\ref{l0}}(n,K)\sqrt{\mu(E)} + C_{\ref{17}}(n,K)|\hat{\alpha}_2|\\
&\leq C_{\ref{l0}}(n,K)\sqrt{\mu(E)} + C_{\ref{17}}(n,K)\bigg(1+\frac{2}{k(n)}\mu(E)\bigg)|\alpha_2|.
\end{align*}
Therefore, we can select $\tilde{c}_{\ref{21}}(n,K), c_{\ref{21}}(n,K)>0$ such that if $\mu(E) \leq c_{\ref{21}}(n,K)$ and $|\alpha_2| \leq \tilde{c}_{\ref{21}}(n,K)$, then 
\begin{equation*} \label{21g}
\min\big\{|\tilde{G} \cap (B_{\frac{1}{2}}(0) \cap \mathcal{C})|, |(B_{\frac{1}{2}}(0) \cap \mathcal{C}) \setminus \tilde{G}| \big\} =  |(B_{\frac{1}{2}}(0) \cap \mathcal{C}) \setminus \tilde{G}|.
\end{equation*}
Thus, using (\ref{21e}) we obtain  
\begin{equation} \label{21h} 
\frac{1}{4} c(n, K) |(B_{\frac{1}{2}}(0) \cap \mathcal{C}) \setminus \tilde{G}| \leq C_{\ref{15}}(n,K) \sqrt{\mu(E)}. \\
\end{equation}
Hence, by (\ref{21h}), (\ref{l2}), and Lemma \ref{17} it follows that   
\begin{align}
\big|\big(&B_{\frac{1}{2}}(0) \cap \mathcal{C}\big) \setminus \big((0,\alpha_2) + K\big)\big| \nonumber \\
&\leq |(B_{\frac{1}{2}}(0) \cap \mathcal{C}) \setminus \tilde{G}| + \big|\tilde{G} \setminus \big((0,\alpha_2)  + K\big)\big| \nonumber\\
& \leq |(B_{\frac{1}{2}}(0) \cap \mathcal{C}) \setminus \tilde{G}| + \big|\tilde{G} \Delta \big((0,\hat{\alpha}_2) + K\big)\big|+\nonumber\\
&\hskip 2in\big|\big((0,\hat{\alpha}_2) + K\big)\Delta \big((0,{\alpha}_2) + K\big)\big| \nonumber\\
&\leq \frac{4C_{\ref{15}}(n,K)}{c(n,K)} \sqrt{\mu(E)} + |(rG) \Delta (\hat{\alpha}+K)|+C_{\ref{17}}(n,K)|\alpha_2-\hat{\alpha}_2|\nonumber \\
&\leq \frac{4C_{\ref{15}}(n,K)}{c(n,K)} \sqrt{\mu(E)} + C_{\ref{l0}}(n,K) \sqrt{\mu(E)}+C_{\ref{17}}(n,K)|\alpha_2-\hat{\alpha}_2| \label{zz2}.
\end{align}
But $|\alpha_2-\hat{\alpha}_2|=|\alpha_2|(r-1)$, and from (\ref{l3}) it readily follows that $|\alpha_2-\hat{\alpha}_2|\leq |\alpha_2| \frac{2}{k(n)}\mu(E) \leq  \tilde{c}_{\ref{21}}(n,K) \frac{2}{k(n)}\mu(E).$ Combining this fact with (\ref{zz2}) yields a positive constant $\tilde{C}(n,K)$ such that   
\begin{equation} \label{z2}
\big|(B_{\frac{1}{2}}(0) \cap \mathcal{C}) \setminus \big((0,\alpha_2) + K\big)\big|\leq \tilde{C}(n,K) \sqrt{\mu(E)}.
\end{equation}   
Next, let $\tilde{K}\subset B_{\frac{1}{2}}(0) \cap \mathcal{C}$ be the bounded, convex set given by Lemma \ref{geo}. We note that since $\alpha \in \mathbb{R}_{+}^n$, (\ref{geo1}) implies 
\begin{equation*} 
\tilde{K} \setminus \big((0,\alpha_2)+\tilde{K}\big)=\tilde{K} \setminus \big((0,\alpha_2)+K\big).
\end{equation*}
Therefore, using Lemma \ref{18} and (\ref{z2}) we have   
\begin{align*}
\min \{c_{\ref{18}}(n&,\tilde{K}), C_{\ref{18}}(n,\tilde{K})|\alpha_2| \} \\
&\leq \big|\big((0,\alpha_2)+\tilde{K}\big) \Delta \tilde{K}\big|=2\big|\tilde{K} \setminus \big((0,\alpha_2)+\tilde{K}\big)\big|\\
&=2\big|\tilde{K} \setminus \big((0,\alpha_2)+K\big)\big| \leq 2\big|(B_{\frac{1}{2}}(0) \cap \mathcal{C}) \setminus \big((0,\alpha_2)+K\big)\big| \\
&\leq 2\tilde{C}(n,K) \sqrt{\mu(E)}.  
\end{align*}
Thus, for $c_{\ref{21}}(n,K)$ sufficiently small we can take $C_{\ref{21}}(n,K)= \frac{2\tilde{C}(n,K)}{C_{\ref{18}}(n,\tilde{K})}$ to conclude the proof.     
\end{proof}

\begin{cor} \label{22}
Let $E \in \mathcal{D}$ with $|E|=|K|$, $c_{\ref{21}}(n,K)$ and $\tilde{c}_{\ref{21}}(n,K)$ be as in Proposition \ref{21}, and $\alpha=\alpha(E)=(\alpha_1,\alpha_2) \in \mathbb{R}^{k} \times \mathbb{R}^{n-k}$ be as in Lemma \ref{l0}. Then there exists a positive constant $C_{\ref{22}}(n,K)$ such that if $\mu(E) \leq c_{\ref{21}}(n,K)$ and $|\alpha_2| \leq \tilde{c}_{\ref{21}}(n,K)$, then $|(E-(\alpha_1,0)) \Delta K| \leq C_{\ref{22}}(n,K) \sqrt{\mu(E)}.$   
\end{cor}

\begin{proof}
Note that by Proposition \ref{21} we obtain $|\alpha_2| \leq C_{\ref{21}}(n,K) \sqrt{\mu(E)}$. Next, by applying Lemma \ref{17}  and (\ref{16a}) we have 
\begin{align*}
|(E-(\alpha_1,0)) \Delta K| &\leq |E \Delta (\alpha+K)|+|((0,\alpha_2)+K) \Delta K| \\
&\leq C_{\ref{l0}}(n,K) \sqrt{\mu(E)}+ C_{\ref{17}}(n,K)|\alpha_2|\\ 
&\leq (C_{\ref{l0}}(n,K)+C_{\ref{17}}(n,K)C_{\ref{21}}(n,K))\sqrt{\mu(E)}.  
\end{align*}
Therefore, we may take $C_{\ref{22}}(n,K)=C_{\ref{l0}}(n,K)+C_{\ref{17}}(n,K)C_{\ref{21}}(n,K)$ to conclude the proof. 
\end{proof}

\subsection{Reduction step}
In Proposition \ref{29} below, we refine Corollary \ref{22}. Namely, we show that if $\mu(E)$ is small enough, then the assumption on the size of $\alpha_2$ is superfluous. However, to prove Proposition \ref{29} we need to reduce the problem to the case when $\alpha_2 \in \overline{\mathcal{\tilde{C}}} \subset \mathbb{R}^{n-k}$ (recall $\mathcal{C} = \mathbb{R}^k \times \mathcal{\tilde{C}}$). This is the content of Lemma \ref{20}. For arbitrary $y \in \mathbb{R}_{+}^{n-k} \setminus \mathcal{\tilde{C}}$, decompose $y$ as
\begin{equation} \label{18a}
y= y^c + y^p,
\end{equation}
where $y^c \in \partial \mathcal{\tilde{C}}$ is the closest point on the boundary of the cone $\mathcal{\tilde{C}}$ to $y$ and $y^p := y-y^c$ (see Figure 3). Note that $y^p$ is perpendicular to $y^c$. 
\begin{figure}[h!] 
\centering 
\includegraphics[width=.3 \textwidth]{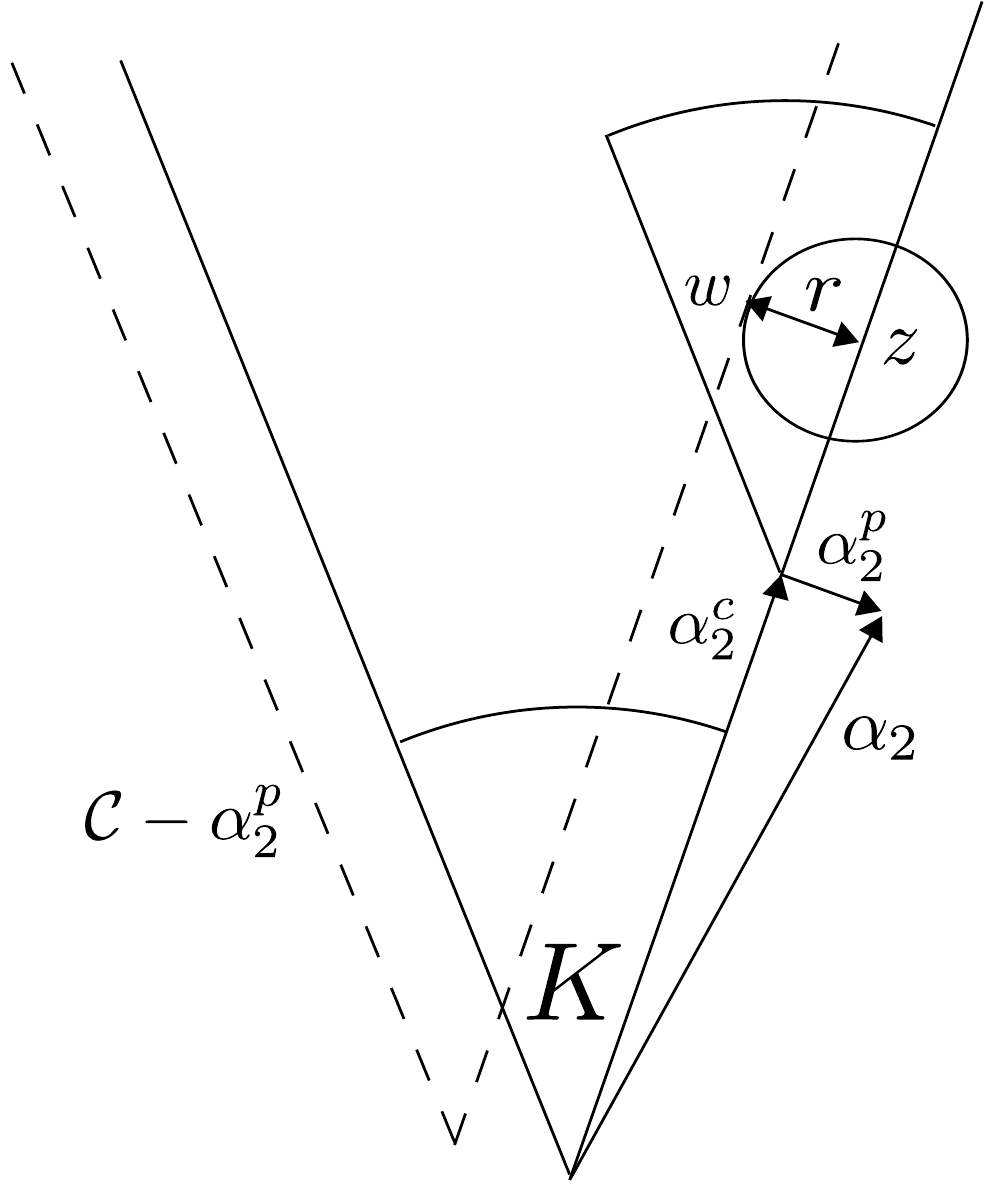}
\caption{Control of $\alpha_2^p$.}
\end{figure}
\begin{lem} \label{20} 
Let $E \in \mathcal{D}$ with $|E|=|K|$, and let $\alpha=\alpha(E)=(\alpha_1,\alpha_2) \in \mathbb{R}^k \times \mathbb{R}^{n-k}$ be as in Lemma \ref{l0}. There exist constants $c_{\ref{20}}(n,K), C_{\ref{20}}(n,K)>0$ such that if $\mu(E) \leq c_{\ref{20}}(n,K)$ and $\alpha_2 \in \mathbb{R}_{+}^{n-k} \setminus \mathcal{\tilde{C}}$, then
$|\alpha_2^p| \leq C_{\ref{20}}(n,K) \mu(E)^{\frac{1}{2n}}$. 
\end{lem}

\begin{proof}
Firstly, observe that 
\begin{align*}
|((0,\alpha_2^c)+K) \setminus (\mathcal{C}-(0,\alpha_2^p))| &= |((0,\alpha_2)+K) \setminus \mathcal{C}|\\
&\leq |(\alpha+K) \setminus (\mathcal{C}+(\alpha_1,0))|=|(\alpha+K) \setminus \mathcal{C}|\\
&\leq |(\alpha+K) \setminus E| \leq C_{\ref{l0}}(n,K) \sqrt{\mu(E)}.
\end{align*}
Since $(0,\alpha_2^c) \in \partial \mathcal{C}$, it follows that $\partial ((0,\alpha_2^c)+K)$ has a nontrivial intersection with $\partial \mathcal{C}$. Let  $$z:=\frac{1}{2}\bigg((0, \alpha_2^c) + \big(0,\alpha_2^c \sup \{t >0: (0,t\alpha_2^c )\in \partial ((0,\alpha_2^c)+K)\}\big)\bigg),$$ and note that, by convexity, $z \in \partial ((0,\alpha_2^c)+K) \cap \partial \mathcal{C}.$ Next, pick $r=|\alpha_2^p|$. Observe that $r$ is the smallest radius for which $B_r(z) \cap \partial (\mathcal{C}-(0,\alpha_2^p)) \neq \emptyset$ so that it contains some $w \in \mathbb{R}^n$ (see Figure 3).
\noindent Since $\mathcal{C}$ is convex, there exists a constant $c_0(n, K)>0$ such that $|B_r(z) \cap ((0,\alpha_2^c)+K)| \geq c_0(n,K) r^n$. But $B_r(z) \cap ((0,\alpha_2^c)+K) \subset ((0,\alpha_2^c)+K) \setminus (C-(0,\alpha_2^p))$. Therefore, $r^n \leq \frac{C_{\ref{l0}}(n,K)}{c_0(n,K)}\sqrt{\mu(E)}$, and since $r=|\alpha_2^p|$ we have that $|\alpha_2^p| \leq \big(\frac{C_{\ref{l0}}(n)}{c_0(n,K)}\big)^{\frac{1}{n}} \mu(E)^{\frac{1}{2n}}$. 
%Therefore, by taking $\mu(E)$ small enough we can make $|\alpha_2^p|$ as small as we wish. This fact and the convexity of $\partial \mathcal{C}$ yield the existence of a radius $r_0=r_0(n,K)>0$ and a constant $c_{\ref{20}}(n,K)>0$ such that a spherical cone centered at $w$ with base radius $r_0$ and height $|\alpha_2^p|$ is contained in $((0,\alpha_2^c)+K) \setminus (C-(0,\alpha_2^p))$, provided that $\mu(E) \leq c_{\ref{20}}(n,K)$ (see Figure 4). 
%This implies, $$\omega_{n-1} r_0^{n-1} \frac{1}{n}|\alpha_2^p| \leq |((0,\alpha_2^c)+K) \setminus (C-(0,\alpha_2^p))|  \leq C_{\ref{l0}}(n,K) \sqrt{\mu(E)},$$ and we take $\displaystyle C_{\ref{20}}(n,K)=\frac{C_{\ref{l0}}(n,K)n}{\omega_{n-1}r_0^{n-1}}$. 
\end{proof}
%\begin{figure}[h!] 
%\centering 
%\includegraphics[width=.5 \textwidth]{alpha2}
%\caption{Convexity allows us to squeeze in a spherical cone in $((0,\alpha_2^c)+K) \setminus (C-(0,\alpha_2^p))$.}
%\end{figure}
\begin{prop} \label{29}
Let $E \in \mathcal{D}$ with $|E|=|K|$, and let $\alpha=\alpha(E)=(\alpha_1,\alpha_2) \in \mathbb{R}^k \times \mathbb{R}^{n-k}$ be as in Lemma \ref{l0}. Then there exist $c_{\ref{29}}(n,K)>0$ such that if $\mu(E) \leq c_{\ref{29}}(n,K)$, then $|\alpha_2| \leq \tilde{c}_{\ref{21}}(n,K)$ with $\tilde{c}_{\ref{21}}(n,K)$ as in Proposition $\ref{21}$.   
\end{prop}

\begin{proof}
If $\alpha \in \mathbb{R}_{-}^n$, then the result follows from Lemma \ref{19}. If $\alpha \in \mathbb{R}_{+}^{n}$, then write $\alpha_2=\alpha_2^p+\alpha_2^c$ as in (\ref{18a}) with the understanding that $\alpha_2 \in \mathcal{\overline{\tilde{C}}}$ if and only if $\alpha_2^p=0$. In the case where $|\alpha_2^p|>0$ (i.e. $\alpha_2 \in \mathbb{R}_{+}^{n-k} \setminus \mathcal{\tilde{C}}$), thanks to Lemma \ref{20}, we have $|\alpha_2^p| \leq C_{\ref{20}}(n,K) \mu(E)^{\frac{1}{2n}}$. Therefore, it suffices to prove that for $\mu(E)$ sufficiently small, $|\alpha_2^c|\leq\frac{1}{2}\tilde{c}_{\ref{21}}(n,K)$. We split the proof into three steps. The idea is as follows: firstly, we assume by contradiction that $|\alpha_2^c|>\frac{1}{2}\tilde{c}_{\ref{21}}(n,K)$. This allows us to translate $E$ by a suitable vector $\beta$ so that $(E-\beta) \cap \mathcal{C}$ is a distance $\frac{1}{4}\tilde{c}_{\ref{21}}(n,K)$ from the origin (see Figure 4). 
\begin{figure}[h!] 
\centering 
\includegraphics[width=.9 \textwidth]{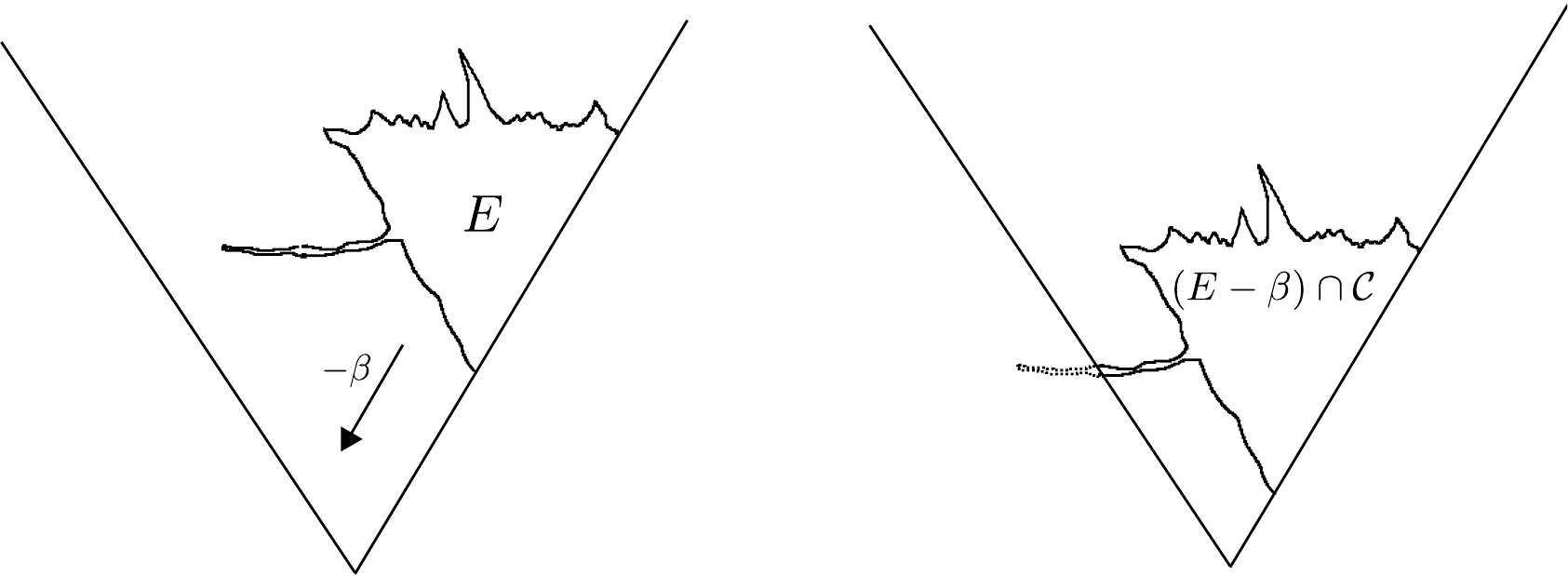}
\caption{If $E$ has small relative deficit but is far away from the origin, we can translate it a little bit and show that the resulting set -- thanks to Proposition \ref{21} -- should in fact be a lot closer to the origin.}
\end{figure}
The second step consists of showing that up to a small mass adjustment, the relative deficit of this new set is controlled by $\mu(E)^{\frac{1}{2n}}$. Lastly, we show that the new set satisfies the hypotheses of Proposition \ref{21}; therefore, we conclude that it should be a lot closer to the origin than it actually is. 
 
\noindent \textbf{Step 1.} Assume for contradiction that $|\alpha_2^c| > \frac{1}{2}\tilde{c}_{\ref{21}}(n,K)$.  Select $\gamma \in (0,1)$ such that for $\beta:=(0,\gamma \alpha_2^c) \in \mathcal{\bar{C}}$ we have $$|(0, \alpha_2^c)-\beta| =(1-\gamma)|\alpha_2^c|= \frac{1}{4}\tilde{c}_{\ref{21}}(n,K).$$  By (\ref{16a}), Lemma \ref{17}, and Lemma \ref{20},   
\begin{align*}
|E \Delta ((\alpha_1,\alpha_2^c)+K)| &\leq |E \Delta (\alpha+K)|+|(\alpha+K) \Delta ((\alpha_1,\alpha_2^c)+ K)| \\
&\leq C_{\ref{l0}}(n,K)\sqrt{\mu(E)} + C_{\ref{17}}(n,K)|\alpha_2^p|\\ 
&\leq C_{\ref{l0}}(n,K)\sqrt{\mu(E)} + C_{\ref{17}}(n,K)C_{\ref{20}}(n,K) \mu(E)^{\frac{1}{2n}}.
\end{align*}
Next, we set $\tilde{E}:=E-(\alpha_1,0)$ and $\tilde{C}(n,K):= C_{\ref{l0}}(n,K)+C_{\ref{17}}(n,K)C_{\ref{20}}(n,K)$ so that
\begin{equation} \label{29ca}
|\tilde{E} \Delta ((0,\alpha_2^c)+K)| \leq \tilde{C}(n,K) \mu(E)^{\frac{1}{2n}}.
\end{equation}
Let $F=t ((\tilde{E}-\beta) \cap \mathcal{C}))$ where $t \geq 1$ is chosen to satisfy $|F|=|\tilde{E}|$. Therefore, 
\begin{equation} \label{29b}
|F|=|\tilde{E}|=|\tilde{E}-\beta|=|(\tilde{E}-\beta) \cap \mathcal{C}|+|(\tilde{E}-\beta) \setminus \mathcal{C}|.
\end{equation}
Now let us focus on the second term on the right side of (\ref{29b}): using (\ref{29ca}),
\begin{align}
|(\tilde{E}-\beta) \setminus \mathcal{C}|&=|\tilde{E} \setminus (\mathcal{C}+\beta)| \nonumber\\
&\leq |\tilde{E} \setminus ((0,\alpha_2^c)+K)|+|((0,\alpha_2^c)+K) \setminus (\mathcal{C}+\beta)| \nonumber\\
&\leq \tilde{C}(n,K)\mu(E)^{\frac{1}{2n}}+|((0,\alpha_2^c)-\beta)+K) \setminus \mathcal{C}|. \label{29c}
\end{align}
But, $(0,\alpha_2^c) - \beta = (0,(1-\gamma)\alpha_2^c) \in \mathcal{\bar{C}}$, therefore, $((0,\alpha_2^c)-\beta)+K \subset \mathcal{C}$, and hence, $|(((0,\alpha_2^c)-\beta)+K) \setminus \mathcal{C}|=0$. Thus, combining the previous fact with (\ref{29b}) and (\ref{29c}),   
\begin{equation} \label{30}
|F|-|(\tilde{E}-\beta) \cap \mathcal{C}| \leq \tilde{C}(n,K)\mu(E)^{\frac{1}{2n}}.
\end{equation}

\noindent \textbf{Step 2.} From the definition of $F$ and $(\ref{30})$, we deduce $$(t^{n}-1)|(\tilde{E}-\beta) \cap \mathcal{C}| \leq \tilde{C}(n,K)\mu(E)^{\frac{1}{2n}},$$ so that for $\mu(E)^{\frac{1}{2n}} \leq \frac{|K|}{2\tilde{C}(n,K)}$, by $(\ref{30})$ again and the fact that $|F|=|K|$, 
\begin{equation} \label{fin5} 
t \leq \biggl(1+\frac{2\tilde{C}(n,K)}{|K|} \mu(E)^{\frac{1}{2n}} \biggr)^{\frac{1}{n}}.
\end{equation}
Since $\mathcal{C}$ is a convex cone, it follows that $\frac{1}{t} \mathcal{C}= \mathcal{C}$ and $\beta+\mathcal{C} \subset \mathcal{C}$. Thus,  
\begin{align}
P(F | \mathcal{C})&= t^{n-1}P\bigl( \tilde{E}| \beta+\mathcal{C} \bigr) \leq t^{n-1}P(\tilde{E}| \mathcal{C})=t^{n-1}P(E| \mathcal{C}+(\alpha_1,0)) \nonumber \\ 
&\leq \biggl(1+\frac{2\tilde{C}(n,K)}{|K|} \mu(E)^{\frac{1}{2n}} \biggr)^{\frac{n-1}{n}}P(E | \mathcal{C}) \nonumber \\
&\leq \biggl(1+\frac{2\tilde{C}(n,K)}{|K|} \mu(E)^{\frac{1}{2n}} \biggr)P(E | \mathcal{C}). \label{31c} 
\end{align}
Recall that $P(F| \mathcal{C})=H^{n-1}(\mathcal{F}F \cap \mathcal{C})$ and $P(E| \mathcal{C})=H^{n-1}(\mathcal{F}E \cap \mathcal{C})$ (see Section 2). Upon subtracting $P(B|\mathcal{C})$ from both sides of (\ref{31c}), dividing by $n|K|$ (recall $n|K|=H^{n-1}(\partial B_1 \cap \mathcal{C}))$, and using that $P(E | \mathcal{C})= n|K|\mu(E)+n|K|$ we have
\begin{align*}
\mu(F) &\leq \mu(E) + \frac{2\tilde{C}(n,K)}{n|K|^2} \mu(E)^{\frac{1}{2n}} P(E | \mathcal{C})\\
&= \mu(E) + \frac{2\tilde{C}(n,K)}{|K|} \mu(E)^{\frac{2n+1}{2n}} + \frac{2\tilde{C}(n,K)}{|K|}\mu(E)^{\frac{1}{2n}}.
\end{align*}
Let $w(n,K):= 1+ \frac{4\tilde{C}(n,K)}{|K|}$. Then, assuming without loss of generality that $\mu(E)\leq 1$,  
\begin{equation} \label{33}
\mu(F) \leq w(n,K) \mu(E)^{\frac{1}{2n}}.
\end{equation} 

\noindent \textbf{Step 3.} Using Lemma \ref{17} and (\ref{fin5}), for $\mu(E)$ small enough we have   
\begin{align}
|F &\Delta (((0,\alpha_2^c) - \beta) + K)| \nonumber \\
&\leq |F \Delta t(((0,\alpha_2^c) - \beta) + K)| + |t(((0,\alpha_2^c) - \beta) + K) \Delta (((0,\alpha_2^c) - \beta) + K)| \nonumber \\
&\leq t^n|(\tilde{E}-\beta) \cap \mathcal{C} \Delta (((0,\alpha_2^c) - \beta) + K)| \nonumber\\
&\hskip .5in  + |t(((0,\alpha_2^c) - \beta) + K) \Delta (t((0,\alpha_2^c) - \beta) + K)|\nonumber\\
&\hskip 1in +|(t((0,\alpha_2^c) - \beta) + K) \Delta (((0,\alpha_2^c) - \beta) + K)| \nonumber  \\
&\leq 2|(\tilde{E}-\beta) \cap \mathcal{C} \Delta (((0,\alpha_2^c) - \beta) + K)| + |(tK) \Delta K| \nonumber\\
&\hskip .5in +C_{\ref{17}}(n,K)|(0,\alpha_2^c)-\beta|(t-1) \nonumber \\
&\leq 2|(\tilde{E}-\beta) \cap \mathcal{C} \Delta (((0,\alpha_2^c) - \beta) + K)| + \bar{C}(n,K) \mu(E)^{\frac{1}{2n}} \label{30a}. 
\end{align}
Next, we claim
\begin{equation} \label{31}
|((\tilde{E}-\beta)\cap \mathcal{C}) \Delta (((0,\alpha_2^c)-\beta)+K)| \leq 2\tilde{C}(n,K) \mu(E)^{\frac{1}{2n}}.
\end{equation}
Indeed, from (\ref{29ca}) we deduce that 
\begin{align}
|((\tilde{E}-\beta)\cap &\mathcal{C}) \Delta (((0,\alpha_2^c) - \beta)+K)| \nonumber \\
&=\bigl | \bigl( ((\tilde{E}-\beta)\cap \mathcal{C})+\beta \bigr) \Delta ((0,\alpha_2^c) + K) \bigr | \nonumber \\
&\leq \bigl | \bigl( ((\tilde{E}-\beta)\cap \mathcal{C})+\beta \bigr) \Delta \tilde{E} \bigr| + \bigl | \tilde{E} \Delta ((0,\alpha_2^c) + K) \big | \nonumber \\
&\leq \big|\bigl( ((\tilde{E}-\beta)\cap \mathcal{C})+\beta \bigr) \Delta \tilde{E} \bigr| + \tilde{C}(n,K)\mu(E)^{\frac{1}{2n}}.\label{31d}
\end{align}
But since $((\tilde{E}-\beta)\cap \mathcal{C}) + \beta \subset \tilde{E}$,
\begin{align}
\bigl | \bigl ( ((\tilde{E}-\beta)\cap \mathcal{C})+\beta \bigr) \Delta \tilde{E} \bigr | &=\bigl | \tilde{E} \setminus \bigl( ((\tilde{E}-\beta)\cap \mathcal{C})+\beta \bigr) \bigr | \nonumber\\
&=\bigl | (\tilde{E}-\beta) \setminus (\tilde{E}-\beta)\cap \mathcal{C} \bigr |\nonumber\\
&=|(\tilde{E}-\beta) \setminus \mathcal{C}| =|\tilde{E} \setminus (\beta + \mathcal{C})|\nonumber\\ 
&\leq |\tilde{E} \setminus ((0,\alpha_2^c)+K)| + |((0,\alpha_2^c)+K) \setminus (\beta+\mathcal{C})|\nonumber\\
&\leq \tilde{C}(n,K)\mu(E)^{\frac{1}{2n}} + |(((0,\alpha_2^c) - \beta)+K) \setminus \mathcal{C}|. \label{31e}
\end{align}
As before, $|(((0,\alpha_2^c) - \beta)+K) \setminus \mathcal{C}|=0$ (since $((0,\alpha_2^c)-\beta)+K \subset \mathcal{C}$). Therefore, (\ref{31d}) and (\ref{31e}) imply the claim (i.e. (\ref{31})).
Furthermore, by using (\ref{30a}) and $(\ref{31})$, it follows that for some constant $\tilde{w}(n,K)$,     
\begin{equation} \label{31a}
|F \Delta (((0,\alpha_2^c) - \beta) + K)| \leq \tilde{w}(n,K)\mu(E)^{\frac{1}{2n}}.
\end{equation}
Next, let $\alpha(F)$ be the translation as in Lemma \ref{l0} for the set $F \subset \mathcal{C}$,  so that $|F \Delta (\alpha(F)+K)| \leq C_{\ref{l0}}(n,K) \sqrt{\mu(F)}$. 
By Lemma $\ref{18}$ and $(\ref{31a})$,  
\begin{align}
\min\{c_{\ref{18}}(n,K)&, C_{\ref{18}}(n,K)|((0,\alpha_2^c)-\beta)-\alpha(F)| \} \nonumber \\
&\leq |(((0,\alpha_2^c)-\beta)+K) \Delta (\alpha(F)+K)| \nonumber \\
&\leq |(((0,\alpha_2^c)-\beta)+K) \Delta F| + |F \Delta (\alpha(F)+K)| \nonumber \\
&\leq \tilde{w}(n,K)\mu(E)^{\frac{1}{2n}} + C_{\ref{l0}}(n,K)\sqrt{\mu(F)}. \label{dm}
\end{align}
Moreover, (\ref{33}) and (\ref{dm}) imply that if $\mu(E)$ is sufficiently small, then there exists a constant $w_2(n,K)$ so that
\begin{align}
|\alpha_2(F)|\leq |\alpha(F)| &\leq |(0,\alpha_2^c)-\beta|+w_2(n,K)\mu(E)^{\frac{1}{4n}}\nonumber\\
&= \frac{1}{4}\tilde{c}_{\ref{21}}(n,K)+w_2(n,K)\mu(E)^{\frac{1}{4n}} \label{33l}, 
\end{align}
and 
\begin{equation} \label{33ne}
|\alpha_1(F)| \leq w_2(n,K) \mu(E)^{\frac{1}{4n}}
\end{equation}     
(since $|\alpha_1(F)| \leq |((0,\alpha_2^c)-\beta)-\alpha(F)|$).
Furthermore, using (\ref{33l}) and (\ref{33}), we deduce that for $\mu(E)$ small enough $$|\alpha_2(F)| \leq \tilde{c}_{\ref{21}}(n,K), \hskip .1in \mu(F) \leq c_{\ref{21}}(n,K),$$ where $c_{\ref{21}}$ is as in Proposition \ref{21}. Thus, by applying Proposition \ref{21} to $F$ and using (\ref{33}) again, it follows that 
\begin{equation} \label{34}
|\alpha_2(F)| \leq C_{\ref{21}}(n,K) \sqrt{\mu(F)} \leq C_{\ref{21}}(n,K) \sqrt{w(n,K)} \mu(E)^{\frac{1}{4n}}.
\end{equation} 
Combining (\ref{dm}), (\ref{33ne}), and (\ref{34}) we obtain 
\begin{align*} 
\frac{1}{4}\tilde{c}_{\ref{21}}(n,K)&=|(0,\alpha_2^c)-\beta| \leq |\alpha(F)|+w_2(n,K)\mu(E)^{\frac{1}{4n}}\\
&\leq |\alpha_2(F)|+2w_2(n,K)\mu(E)^{\frac{1}{4n}}\\
&\leq \bigg(C_{\ref{21}}(n,K) \sqrt{w(n,K)} +2w_2(n,K)\bigg)\mu(E)^{\frac{1}{4n}},
\end{align*}
which is impossible if $\mu(E)$ is sufficiently small. This concludes the proof. 
\end{proof}

We are now in a position to prove Theorem \ref{t1}. Firstly, we assume that $|E|=|K|$. Indeed, let $c_{\ref{21}}$ and $c_{\ref{29}}$ be the constants given in Proposition \ref{21} and \ref{29}, respectively, and set $c(n,K):=\min\{c_{\ref{21}}(n,K), c_{\ref{29}}(n,K)\}$. If $\mu(E) \leq c(n,K)$, then it follows from Proposition \ref{29} and Corollary \ref{22} that $$\frac{|(E-(\alpha_1,0)) \Delta K|}{|K|} \leq \frac{C_{\ref{22}}(n,K)}{|K|} \sqrt{\mu(E)}.$$ Let $\bar{C}(n,K):=\frac{C_{\ref{22}}(n,K)}{|K|}$ and suppose now that $|E| \neq |K|$. Pick $t>0$ such that $|tE|=|K|$ and apply the previous estimate to the set $tE$ to obtain $$\frac{|(tE-(\alpha_1(tE),0)) \Delta K)|}{|tE|} \leq \bar{C}(n,K)\sqrt{\mu(tE)} = \bar{C}(n,K) \sqrt{\mu(E)},$$ and this implies $$\frac{|(E-(\frac{\alpha_1(tE)}{t},0)) \Delta (\frac{1}{t}K)|}{|E|} \leq \bar{C}(n,K)\sqrt{\mu(E)}.$$ Since $s=\frac{1}{t}$, this yields the theorem for the case when $\mu(E)\leq c(n,K).$ If $\mu(E) > c(n,K)$, then $$\frac{|E \Delta (sK)|}{|E|} \leq 2 \leq \frac{2}{\sqrt{c(n,K)}} \sqrt{\mu(E)}.$$ Therefore, we obtain the theorem with $C(n,K)=\min \bigg\{\bar{C}(n,K), \frac{2}{\sqrt{c(n,K)}} \bigg\}$.

\appendix 

\section{Proofs of the technical lemmas}

\begin{proof}[\bf{Proof of Lemma \ref{18}.}]
For $y \in \mathbb R^n$ let $$f(y):=|(A+y) \Delta A|.$$ Note that $f(y) = \int_{\mathbb{R}^n} |\mathbf{1}_{(A+y)}(x) - \mathbf{1}_{A}(x)| dx.$ 

Our strategy is as follows:
first, we show that the incremental ratios of $f$ at $y=0$ 
have a positive lower bound.
Then we prove that $f(y)$ is uniformly bounded away from zero when $y$ is away from zero.
These two facts readily yield the result. \\

\noindent \textbf{Step 1.} We claim there exists $s=s(n, A)>0$ such that  
\begin{equation} \label{r0}
C_{\ref{18}}(n,A):=\inf_{0 \leq |y| \leq s}\frac{f(y)}{|y|}>0 .
\end{equation} 
Indeed, let $y_k \in \mathbb{R}^n$ be any sequence converging to $0$.
By Lemma \ref{17}
we know that the family of measure $\mu_k$ defined by
$$
\mu_k:=\frac{\mathbf{1}_{A+y_k} - \mathbf{1}_{A}}{|y_k|} 
$$ 
satisfy $|\mu_k|(\mathbb{R}^n) \leq C_{\ref{17}}$.
Moreover,
up to choosing a subsequence (which we do not relabel) so that $y_k/|y_k| \to w$
for some $w \in \mathbb{S}^{n-1}$, it is immediate to check that $\mu_k$ converge weakly to 
$D \mathbf{1}_A \cdot w.$
Hence, by the lower semicontinuity of the total variation
(see for instance \cite[Corollary 1.60]{AFP}),
$$
\liminf_{k \to \infty} \frac{|(A+y_k) \Delta A|}{|y_k|} =
\liminf_{k \to \infty} |\mu_k|(\mathbb{R}^n) \geq |D \mathbf{1}_A \cdot w|(\mathbb{R}^n).
$$
We now observe that, again by the lower semicontinuity of the total variation,
the right hand side attains a minimum for some $\bar w \in \mathbb{S}^{n-1}$.
Hence, by the arbitrariness of $y_k$,
$$
\liminf_{|y| \to 0} \frac{f(y)}{|y|}=\liminf_{|y| \to 0} \frac{|(A+y) \Delta A|}{|y|} \geq |D \mathbf{1}_A \cdot \bar w|(\mathbb{R}^n).
$$
To conclude it suffices to observe that $|D \mathbf{1}_A \cdot \bar w|(\mathbb{R}^n)>0$,
as otherwise $A=A+t\bar w$ for any $t \in \mathbb{R}$ (up to sets of measure zero), which contradicts the boundedness of $A$.

\noindent \textbf{Step 2.}       
We claim that there exists $c_{\ref{18}}(n,A)>0$ such that
\begin{equation} \label{r0bis}
f(y) \geq c_{\ref{18}}(n,A) \qquad \forall\, |y| \geq s(n,A),
\end{equation} 
with $s(n,A)$ as in Step 1. The proof is by compactness: if $|y|\geq {\rm diam}(A)$ then $f(y) =2|A|>0$.
On the other hand, by continuity $f$ attains a minimum over the compact set
$\{y:s(n,A) \leq |y|\leq {\rm diam}(A)\}$. Let $\bar y$ be a vector where such
a minimum is attained.
Then this minimum cannot be zero as otherwise
$A=A+\bar y$ (up to a set of measure zero).
By iterating the estimate, this implies that $A=A+k\bar y$ for any $k \in \mathbb Z$,
contradicting again the boundedness of $A$.

\end{proof}

\begin{proof}[\bf{Proof of Lemma \ref{21a}.}]
Let $k(n)=\frac{2-2^{\frac{n-1}{n}}}{3}$. If $\mu(E) \leq \frac{k(n)^2}{8}:=c_{\ref{21a}}(n)$, then by \cite[Theorem 3.4]{Fi} there exists a set of finite perimeter $G \subset E$ satisfying 
\begin{equation} \label{lma}
|E \setminus G| \leq \frac{\delta_{K_0}(E)}{k(n)}|E|,
\end{equation}
\begin{equation} \label{lmb}
\tau(G) \geq 1 + \frac{m_{K_0}}{M_{K_0}}k(n). 
\end{equation}
We claim that $G$ is the desired set. Indeed, since $\delta_{K_0}(E) \leq \mu(E)$ (see Corollary \ref{8}), (\ref{lma}) and (\ref{lmb}) yield (\ref{23}) and (\ref{24}); therefore, it remains to prove (\ref{24d}). From the construction of $G$ in \cite[Theorem 3.4]{Fi}, we have $G=E\setminus F_{\infty}$, where $F_{\infty} \subset E$ is the maximal element given by \cite[Lemma 3.2]{Fi} that satisfies
\begin{equation} \label{24e}
P_{K_0}(F_{\infty}) \leq \bigg (1+\frac{m_{K_0}}{M_{K_0}}k(n) \bigg) \int_{\mathcal{F} F_{\infty} \cap \mathcal{F} E} ||\nu_E(x)||_{K_0*} dH^{n-1}(x).
\end{equation}
To prove (\ref{24d}), we first claim that for some positive constant $C(n,K)$, 
\begin{equation} \label{25}
H^{n-1}(\mathcal{F} G \cap \mathcal{C}) \leq H^{n-1}(\mathcal{F} E \cap \mathcal{C})+C(n,K)\mu(E).
\end{equation} 
Note from the definitions that 
\begin{equation} \label{25d}
P_{K_0}(E) = n|K|\delta_{K_0}(E)+n|K|^{\frac{1}{n}}|E|^{\frac{n-1}{n}}.
\end{equation}
Moreover, by \cite[Equation (2.10)]{Fi} and \cite[Equation (2.11)]{Fi} we may write $$P_{K_0}(G) = \int_{\mathcal{F} G \cap \mathcal{F} E} ||\nu_E(x)||_{K_0*} dH^{n-1}(x)+\int_{\mathcal{F} G \cap E^{(1)}} ||\nu_G(x)||_{K_0*} dH^{n-1}(x).$$ Therefore, 
\begin{align}
P_{K_0}(E)&= \int_{\mathcal{F} G \cap \mathcal{F} E} ||\nu_E(x)||_{K_0*} dH^{n-1}(x)+\int_{\mathcal{F} F_{\infty} \cap \mathcal{F} E} ||\nu_E(x)||_{K_0*} dH^{n-1}(x) \nonumber\\
&= P_{K_0}(G) - \int_{\mathcal{F} G \cap E^{(1)}} ||\nu_G(x)||_{K_0*} dH^{n-1}(x) \nonumber \\
&\hskip 2in +\int_{\mathcal{F} F_{\infty} \cap \mathcal{F} E} ||\nu_E(x)||_{K_0*} dH^{n-1}(x). \label{24f}  
\end{align}  
Next, we note that $\mathcal{F} F_{\infty} \cap E^{(1)}=\mathcal{F} G \cap E^{(1)},$ and by \cite[Lemma 2.2]{Fi}, $\nu_G=-\nu_{F_{\infty}}$ at $H^{n-1}$ -- a.e. point of $\mathcal{F} F_{\infty} \cap E^{(1)} .$ Furthermore, taking into account (\ref{10a}) and (\ref{24e}), we have
\begin{align}
\int_{\mathcal{F} G \cap E^{(1)}} ||\nu_G(x)||_{K_0*} &dH^{n-1}(x) \nonumber \\ 
&=\int_{\mathcal{F} F_{\infty} \cap E^{(1)}} ||-\nu_{F_{\infty}}(x)||_{K_0*} dH^{n-1}(x)\nonumber\\
&\leq \frac{M_{K_0}}{m_{K_0}} \int_{\mathcal{F} F_{\infty} \cap E^{(1)}} ||\nu_{F_{\infty}}(x)||_{K_0*} dH^{n-1}(x)\nonumber\\
&\leq \frac{M_{K_0}}{m_{K_0}} \frac{m_{K_0}}{M_{K_0}}k(n)  \int_{\mathcal{F} F_{\infty} \cap \mathcal{F}E} ||\nu_{F_{\infty}}(x)||_{K_0*} dH^{n-1}(x)\nonumber\\
&=k(n) \int_{\mathcal{F} F_{\infty} \cap \mathcal{F} E} ||\nu_{F_{\infty}}(x)||_{K_0*} dH^{n-1}(x). \label{24g}
\end{align}
Hence, (\ref{24f}) and (\ref{24g}) yield (observe that $\nu_E=\nu_{F_\infty}$ on $\mathcal{F} F_\infty \cap \mathcal{F} E$)
\begin{equation} \label{fin6}
P_{K_0}(E) \geq P_{K_0}(G)+(1-k(n))\int_{\mathcal{F} F_{\infty} \cap \mathcal{F} E} ||\nu_E(x)||_{K_0*} dH^{n-1}(x).
\end{equation} 
By the anisotropic isoperimetric inequality (see \cite[Theorem 2.3]{Fi} or (\ref{las2})), $$P_{K_0}(G) \geq n|K|^{\frac{1}{n}}|G|^{\frac{n-1}{n}}.$$ Moreover, by (\ref{lma}), 
\begin{equation} \label{fin7}
|E|-|G| \leq \frac{\mu(E)}{k(n)}|E|.
\end{equation}
Thus, 
\begin{align}
P_{K_0}(G)&\geq n|K|^{\frac{1}{n}}|G|^{\frac{n-1}{n}} \geq n|K|^{\frac{1}{n}}\biggl( |E|-\frac{\mu(E)}{k(n)}|E|\biggr)^{\frac{n-1}{n}}\nonumber\\
&\geq n|K|^{\frac{1}{n}}|E|^{\frac{n-1}{n}}\biggl(1-\frac{\mu(E)}{k(n)}\biggr).\label{24h} 
\end{align}
Combining (\ref{25d}), (\ref{fin6}), and (\ref{24h}) it follows that
\begin{align*}
n|K|\delta_{K_0}(E)+n|K|^{\frac{1}{n}}|E|^{\frac{n-1}{n}} &\geq n|K|^{\frac{1}{n}}|E|^{\frac{n-1}{n}}\biggl(1-\frac{\mu(E)}{k(n)}\biggr)\\
&+(1-k(n))\int_{\mathcal{F} F_{\infty} \cap \mathcal{F} E} ||\nu_E(x)||_{K_0*} dH^{n-1}(x).
\end{align*}
Therefore (recall $\delta_{K_0}(E) \leq \mu(E)$ and $|E|=|K|$), 
\begin{equation} \label{24i}
\int_{\mathcal{F} F_{\infty} \cap \mathcal{F} E} ||\nu_E(x)||_{K_0*} dH^{n-1}(x) \leq \frac{n|K|(1+k(n))}{k(n)(1-k(n))} \mu(E)
\end{equation}
Using the definition of $m_{K_0}$, (\ref{24g}), and (\ref{24i}) we obtain 
\begin{align*}
H^{n-1}(\mathcal{F} G \cap \mathcal{C})&= H^{n-1}(\mathcal{F} G \cap \mathcal{F} E \cap \mathcal{C})+H^{n-1}(\mathcal{F} G \cap E^{(1)})\\
&\leq H^{n-1}(\mathcal{F} E \cap \mathcal{C}) + H^{n-1}(\mathcal{F} G \cap E^{(1)})\\
&\leq H^{n-1}(\mathcal{F} E \cap \mathcal{C}) + \frac{1}{m_{K_0}}\int_{\mathcal{F} G \cap E^{(1)}} ||\nu_G(x)||_{K_0*} dH^{n-1}(x)\\  
&\leq H^{n-1}(\mathcal{F} E \cap \mathcal{C}) + \frac{1}{m_{K_0}}k(n) \int_{\mathcal{F} F_{\infty} \cap \mathcal{F} E} ||\nu_{F_{\infty}}(x)||_{K_0*} dH^{n-1}(x)\\
&\leq H^{n-1}(\mathcal{F} E \cap \mathcal{C}) + \frac{1}{m_{K_0}}\frac{n|K|(1+k(n))}{(1-k(n))} \mu(E), 
\end{align*}
and this proves our claim (i.e. (\ref{25})). Our next task is to use (\ref{25}) in order to prove (\ref{24d}), thereby finishing the proof of the lemma. Let $r>0$ be such that $|rG|=|E|$. Note that by (\ref{25}),  
\begin{align}
\mu(G)&=\mu(rG)= \frac{H^{n-1}(\mathcal{F} (rG) \cap \mathcal{C}) - H^{n-1}(\partial B_1 \cap \mathcal{C})}{H^{n-1}(\partial B_1 \cap \mathcal{C})}\nonumber\\
&= \frac{r^{n-1}H^{n-1}(\mathcal{F} G \cap \mathcal{C}) - H^{n-1}(\partial B_1 \cap \mathcal{C})}{H^{n-1}(\partial B_1 \cap \mathcal{C})}\nonumber\\
&\leq \frac{r^{n-1}\bigl(H^{n-1}(\mathcal{F} E \cap \mathcal{C})+C(n,K)\mu(E)\bigr) - H^{n-1}(\partial B_1 \cap \mathcal{C})}{H^{n-1}(\partial B_1 \cap \mathcal{C})}. \label{28k} 
\end{align}
But since $\mu(E) \leq \frac{k(n)^2}{8}$ and $k(n)\leq \frac{1}{2}$, we have $\frac{\mu(E)}{k(n)}\leq \frac{k(n)}{8}\leq \frac{1}{16}$ so that two applications of (\ref{fin7}) yield    
\begin{equation} \label{28a}
\frac{|K|}{|G|} \leq 1+ \frac{16}{15} \frac{\mu(E)}{k(n)} \leq 1+ \mu(E)\frac{2}{k(n)},
\end{equation} 
and by using $(\ref{28a})$ we have 
\begin{align}
r^{n-1}&= \biggl(\frac{|K|}{|G|}\biggr)^{\frac{n-1}{n}} \leq \biggl(1+ \mu(E)\frac{2}{k(n)}\biggr)^{\frac{n-1}{n}}\leq 1+ \mu(E)\frac{2(n-1)}{n k(n)}. \label{28l}   
\end{align} 
Upon combining (\ref{28k}) and (\ref{28l}), (\ref{24d}) follows easily.  
\end{proof}

\begin{proof}[\bf{Proof of Lemma \ref{l0}.}] 
Recall that by definition $c_{\ref{21a}}(n)=\frac{k^2(n)}{8}$, where $k(n)=\frac{2-2^{\frac{n-1}{n}}}{3}$. Since $\delta_{K_0}(E) \leq \mu(E)$, by taking $\mu(E) \leq c_{\ref{21a}}(n)$, $\delta_{K_0}(E)$ will be sufficiently small in order for us to assume the setup of \cite[Inequality (3.30)]{Fi}, with the understanding that the set $K$ in the equation corresponds to our $K_0$, and whenever $K$ appears in our estimates, it is the same set that we defined in the introduction (i.e. $K=B_1 \cap \mathcal{C}$). Note that in \cite[Proof of Theorem 1.1]{Fi} the authors dilate the sets $G$ and $E$ by the same factor $r>0$ so that $|rG|=|K_0|=|K|$; however, they still denote the resulting dilated sets by $G$ and $E$. We will keep the scaling factor so that our $rG$ and $rE$ correspond, respectively, to their $G$ and $E$. With this in mind, note that \cite[Inequality (3.30)]{Fi} is valid up to a translation. Indeed, this translation is obtained by applying \cite[Lemma 3.1]{Fi} to the functions $S_i$ and the set $rG$, where $S(x)=\tilde{T}_0(x)-x$, and $\tilde{T}_0$ is the Brenier map between $rG$ and $K_0$. Since $a=\hat{\alpha} -x_0$ in Lemma $\ref{15}$ was obtained by the same exact process, $a$ satisfies \cite[Inequality (3.30)]{Fi}. Thus, by the estimates under \cite[Inequality (3.30)]{Fi} it follows that     
\begin{align*} 
C(n,K)\sqrt{\delta_{K_0}(rG)} &\geq \int_{\mathcal{F}(rG)} \big|1-||x-a||_{K_0} \big| ||\nu_{rG}(x)||_{K_0*}dH^{n-1}(x)\\
&= \int_{\mathcal{F}(rG-a)} \bigl |1-||x||_{K_0} \bigr| ||\nu_{(rG-a)}(x)||_{K_0*}dH^{n-1}(x)\\
&\geq \frac{m_{K_0}}{M_{K_0}}|(rG-a) \setminus K_0|.
\end{align*}  
Therefore, we have 
\begin{align}
|(rG) \Delta (\hat{\alpha}+K)|&=|(rG) \Delta (a+K_0)|=2|(rG-a)\setminus K_0| \nonumber \\
&\leq 2\tilde{C}(n,K)\sqrt{\delta_{K_0}(G)}\leq 2\tilde{C}(n,K)\sqrt{\mu(G)}, \label{101}
\end{align}
and this implies  
\begin{align}
|(rE) \Delta (\hat{\alpha}+K)| &\leq |(rE) \Delta (rG)|+|(rG) \Delta (\hat{\alpha}+K)| \nonumber \\
&\leq 2r^n|E \setminus G|+ 2\tilde{C}(n,K)\sqrt{\mu(G)} \label{102}.
\end{align}
Recalling that $|E\setminus G|=|E|-|G| \leq \frac{|E|}{k(n)} \mu(E)$ (see (\ref{23})), $|E|=|K|$, and $\mu(E)$ is small, it readily follows that (see (\ref{28a})) 
\begin{equation} \label{rbs}
r \leq 1+\frac{2\mu(E)}{k(n)},
\end{equation} 
and we obtain (\ref{l3}). Also, (\ref{101}), (\ref{102}), (\ref{rbs}), and $\mu(G) \leq C_{\ref{21a}}(n,K) \mu(E)$ (see (\ref{24d})) imply the existence of a positive constant $C(n,K)$ so that 
\begin{equation} \label{104}
|(rE) \Delta (\hat{\alpha}+K)| \leq C(n,K)\sqrt{\mu(E)},   
\end{equation}
\begin{equation} \label{105}
|(rG) \Delta (\hat{\alpha}+K)| \leq C(n,K) \sqrt{\mu(E)}.  
\end{equation}
Moreover, (\ref{rbs}) and (\ref{104}) imply 
\begin{align*}
|E \Delta (\alpha+K)| &\leq |E \Delta (\alpha+\frac{1}{r}K)|+|(\alpha+\frac{1}{r} K) \Delta (\alpha+K)|\\
&\leq \frac{1}{r^n} |(rE) \Delta (\hat{\alpha}+K)|+|K \setminus \frac{1}{r}K| \\
&\leq C(n,K)\sqrt{\mu(E)}+|K|\bigg(\frac{r-1}{r}\bigg)\\
&\leq C(n,K)\sqrt{\mu(E)}+|K|\frac{2}{k(n)} \mu(E).
\end{align*}
By combining this together with (\ref{105}), we readily obtain (\ref{16a}) and (\ref{l2}).    
\end{proof}

\noindent \textbf{Acknowledgments.} We wish to thank Eric Baer for his careful remarks on a preliminary version of the paper.
We are also grateful to an anonymous referee for his valuable comments and suggestions
that significantly improved the presentation of the paper.
The first author was supported by NSF grant DMS-$0969962$. The second author was supported by an NSF-RTG fellowship for graduate studies at the University of Texas at Austin.  
Any opinions, findings, and conclusions or recommendations expressed in this material are
those of the authors and do not necessarily reflect the views of the US National Science Foundation (NSF).

\signaf
\signei


\begin{thebibliography}{99}
\bibitem{AA} A. Alberti and L. Ambrosio, A geometrical approach to monotone functions in $\mathbb{R}^n$, \textit{Math. Z.} 230 (1999), no. 2, 259-316.
\bibitem{AFP} L. Ambrosio, N. Fusco, and D. Pallara, Functions of bounded variation and free discontinuity problems. Oxford Mathematical Monographs. The Clarendon Press, \textit{Oxford University Press}, New York, 2000.
%\bibitem{Be} F. Bernstein, \"Uber die isoperimetrische Eigenschaft des Kreises auf der Kugeloberfl\"ache und in der Ebene, \textit{Math. Ann.}, 60 (1905), 117-136.
\bibitem{Bia} G. Bianchi and H. Egnell, A note on the Sobolev inequality, \textit{J. Funct. Anal.} 100 (1991), pp. 18-24.
\bibitem{BKK} W. Blaschke, Kreis und Kugel, Berlin. \textit{de Gruyter}, 2nd edn. (1956).
%\bibitem{Bo} T. Bonnesen, \"Uber die isoperimetrische Defizite ebener Figuren, \textit{Math. Ann.}, 91 (1924), 252-268.
\bibitem{Br} Y. Brenier, Polar factorization and monotone rearrangement of vector-valued functions, \textit{Comm. Pure Appl. Math.} 44 (4) (1991) 375-417.
\bibitem{BM} J.E. Brothers and F. Morgan, The isoperimetric theorem for general integrands, \textit{Michigan Math. J.} 41 (1994), no. 3, 419-431.
 \bibitem{BZ}Y.D. Burago and V.A. Zalgaller, Geometric inequalities. Berlin: \textit{Springer Verlag} (1988). Original russian edition: Geometricheskie neravenstva, Leningrad.
\bibitem{IC}I. Chavel, Isoperimetric inequalities. Cambridge Tracts in Mathematics, 145, \textit{Cambridge University Press} (2001).
\bibitem{CF} E. Carlen and A. Figalli, Stability for a GNS inequality and the Log-HLS inequality, with application to the critical mass Keller-Segel equation, preprint, 2011.
\bibitem{Cia} A. Cianchi, N. Fusco, F. Maggi, and A. Pratelli, The sharp Sobolev inequality in quantitative form, \textit{Journal of the European Mathematical Society}, (2009) 11, n.5, 1105-1139. 
\bibitem{Cord} D. Cordero-Erausquin, B. Nazaret, and C. Villani, A new approach to sharp Sobolev and Gagliardo-Nirenberg inequalities, \textit{Adv. Math.} 182, 2 (2004), 307-332.
\bibitem{DP}B. Dacorogna and C.E. Pfister, Wulff theorem and best constant in Sobolev inequality, \textit{J. Math. Pures Appl.} (9) 71 (2) (1992) 97-118.
\bibitem{Di} A. Dinghas, \"Uber einen geometrischen Satz von Wulff f\"ur die Gleichgewichtsform von
Kristallen, \textit{(German) Z. Kristallogr., Mineral. Petrogr.} 105, (1944).
\bibitem{EFT}L. Esposito, N. Fusco, and C. Trombetti, A quantitative version of the isoperimetric inequality: the anisotropic case, \textit{Ann. Sc. Norm. Super. Pisa Cl. Sci.} (5) 4 (2005), no. 4, 619-651.
\bibitem{Fi4} A. Figalli, F. Maggi, On the shape of liquid drops and crystals in the small mass regime, \textit{Arch. Ration. Mech. Anal.}, to appear.
\bibitem{Fi} A. Figalli, F. Maggi, and A. Pratelli, A mass transportation approach to quantitative isoperimetric inequalities, \textit{Invent. Math.}, 182 (2010), no. 1, 167-211.
\bibitem{Fi2} A. Figalli, F. Maggi, and A. Pratelli, A refined Brunn-Minkowski inequality for convex sets,  \textit{Ann. Inst. H. Poincar{\'e} Anal. Non Lineaire}, 26  (2009),  no. 6, 2511-2519.
\bibitem{Fi3} A. Figalli, F. Maggi, and A. Pratelli, A note on Cheeger sets, \textit{Proc. Amer. Math. Soc.}, 137 (2009), 2057-2062.
\bibitem{Fi5} A. Figalli, F. Maggi, and A. Pratelli, Sharp stability theorems for the anisotropic Sobolev and log-Sobolev inequalities on functions of bounded variation. Preprint, 2011.
\bibitem{FM} I. Fonseca and S. M\"uller, A uniqueness proof for the Wulff theorem, \textit{Proc. Roy. Soc. Edinburgh Sect.} A 119 (1991), no. 1-2, 125-136.
%\bibitem{Fu}B. Fuglede, Stability in the isoperimetric problem for convex or nearly spherical domains
%in $R^n$, \textit{Trans. Amer. Math. Soc.}, 314 (1989), 619-638.
\bibitem{morganbook} F. Morgan, Riemannian geometry. A beginner's guide. Second edition. \textit{A. K. Peters, Ltd.}, Wellesley, MA, 1998.
\bibitem{morgan} F. Morgan, The Levy-Gromov isoperimetric inequality in convex manifolds with boundary, \textit{J. Geom. Anal.} 18  (2008),  no. 4, 1053-1057.
\bibitem{NF} N. Fusco, The classical isoperimetric theorem, \textit{Rend. Acc. Sci Fis. Mat. Napoli}, Vol 71, (2004), 63-107.
\bibitem{FMP}N. Fusco, F. Maggi, and A. Pratelli, The sharp quantitative isoperimetric inequality, \textit{Ann. of Math.} 168 (2008), 941-980.
\bibitem{Gu}M.E. Gurtin, On a theory of phase transitions with interfacial energy, \textit{Arch. Rational
Mech. Anal.} 87 (1985), no. 3, 187-212.
%\bibitem{Ha}R.R. Hall, A quantitative isoperimetric inequality in n-dimensional space, \textit{J. Reine Angew. Math.}, 428 (1992), 161-176.
%\bibitem{HHW}R.R. Hall, W.K. Hayman, and A.W. Weitsman, On asymmetry and capacity, \textit{J. d-Analyse Math.}, 56 (1991), 87-123.
\bibitem{He}C. Herring, Some theorems on the free energies of crystal surfaces, \textit{Phys. Rev.} 82 (1951), 87-93.
\bibitem{LP}P. L. Lions and F. Pacella, Isoperimetric inequalities for convex cones, \textit{Proc. Amer. Math. Soc.} 109 (1990) 477-485.
\bibitem{Ma}F. Maggi, Some methods for studying stability in isoperimetric type problems, \textit{Bull. Amer. Math. Soc.}, 45 (2008), 367-408.
\bibitem{McC2} R.J. McCann, A convexity principle for interacting gases, \textit{Adv. Math.} 128 (1) (1997) 153-179.
\bibitem{McC1} R.J. McCann, Existence and uniqueness of monotone measure-preserving maps, \textit{Duke Math. J}. 80 (2) (1995) 309-323.
\bibitem{MS} V.D. Milman and G. Schechtman, Asymptotic theory of finite-dimensional normed spaces.
With an appendix by M. Gromov. Lecture Notes in Mathematics, 1200. \textit{Springer-Verlag},
Berlin, 1986. viii+156 pp.
\bibitem{Mur} C.B. Muratov, Droplet phases in non-local Ginzburg-Landau models with Coulomb repulsion in two dimensions. \textit{Comm. Math. Phys.}  299  (2010),  no. 1, 45-87.
\bibitem{Os}R. Osserman, The isoperimetric inequality, \textit{Bull. Amer. Math. Soc.}, 84 (1978), 1182-1238.
%\bibitem{Ro}R.T. Rockafellar, Convex analysis. \textit{Princeton University Press}, Princeton, NJ, 1997. Reprint of the 1970 original, Princeton Paperbacks.
\bibitem{RR} M. Ritor\'e and C. Rosales, Existence and characterization of regions minimizing perimeter under a volume constraint inside Euclidean cones, \textit{Trans. Amer. Math. Soc.} 356 (2004), no. 11, 4601-4622.
\bibitem{BMrmk} A. Segal, Remark on Stability of Brunn-Minkowski and Isoperimetric Inequalities for Convex Bodies, preprint, 2011.
\bibitem{GT}G. Talenti, The standard isoperimetric theorem, \textit{Handbook of convex geometry}, Vol. A, P.M. Grueber and J.M. Willis, eds., 73-123, Amsterdam, North Holland (1993).
\bibitem{Ty}J.E. Taylor, Crystalline variational problems, \textit{Bull. Amer. Math. Soc.} 84 (1978), no. 4,
568-588.
\bibitem{VS}J. Van Schaftingen, Anisotropic symmetrization, \textit{Ann. Inst. H. Poincar{\'e} Anal. Non Lineaire} 23 (2006), no. 4, 539-565.
\bibitem{Wu}G. Wulff, Zur Frage der Geschwindigkeit des Wachstums und der Aufl\"osung der Kristallfl\"achen, \textit{Z. Kristallogr.} 34, 449-530.

\end{thebibliography}
\end{document}